\documentclass[a4paper,12pt]{article}
\usepackage{amsthm,amsmath,amssymb,bbm,geometry,epsfig,hyperref,
enumerate,comment,nicefrac,listings,latexsym,mathrsfs,mathtools,
rotating,float,subeqnarray,cases,indentfirst,graphicx,subfigure,
amsfonts,booktabs,bm,epstopdf,outlines}

\newtheorem{lemma}{Lemma}[section]
\newtheorem{remark}[lemma]{Remark}

\newtheorem{example}[lemma]{Example}

\newtheorem{algo}[lemma]{Framework}

\newcommand{\smallsum}{\textstyle\sum}

\providecommand{\N}{{\ensuremath{\mathbb{N}}}}

\providecommand{\R}{{\ensuremath{\mathbb{R}}}}
\providecommand{\E}{{\ensuremath{\mathbbm{E}}}}

\renewcommand{\P}{\mathbbm{P}}

\title{Deep learning-based numerical methods for high-dimensional 
parabolic partial differential equations and backward stochastic 
differential equations}

\author{Weinan E$^1$, Jiequn Han$^2$, and Arnulf Jentzen$^3$
\medskip
\\
\small{$^1$Beijing Institute of Big Data Research (China), 
Princeton University (USA),}
\\
\small{and Peking University (China), 
e-mail: weinan (at) math.princeton.edu}
\smallskip
\\
\small{$^2$Princeton University (USA), 
e-mail: jiequnh (at) princeton.edu}
\smallskip
\\
\small{$^3$ETH Zurich (Switzerland), 
e-mail: arnulf.jentzen (at) sam.math.ethz.ch}}

\begin{document}

\maketitle

\begin{abstract}
We propose a new algorithm 
for solving parabolic partial differential equations (PDEs)
and backward stochastic differential 
equations (BSDEs) in high dimension, by making an analogy between the BSDE 
and reinforcement learning with the gradient of the solution playing
the role of the policy function, and 
the loss function given by 
the error between the prescribed terminal condition 
and the solution of the BSDE. 
The policy function is then approximated by a neural network, 
as is done in deep reinforcement learning.
Numerical results using {\sc TensorFlow} illustrate the efficiency 
and accuracy of the proposed algorithms 
for several $ 100 $-dimensional nonlinear PDEs from physics and finance 
such as the Allen-Cahn equation, the Hamilton-Jacobi-Bellman equation, and 
a nonlinear pricing model for financial derivatives.
\end{abstract}

\tableofcontents

\section{Introduction}

Developing efficient numerical algorithms for high dimensional (say,  hundreds of dimensions) partial
differential equations (PDEs) has been one of the most challenging
tasks in applied mathematics.  As is well-known, the difficulty lies 
in the ``curse of dimensionality'' \cite{MR2641641}, namely,
as the dimensionality grows, the complexity of the algorithms grows exponentially.
For this reason,  there are only a limited number of cases where 
practical high dimensional algorithms have been developed.
For linear parabolic PDEs, one can use the Feynman-Kac formula and Monte Carlo methods to develop efficient algorithms
to evaluate solutions at any given space-time locations.  
For a class of inviscid Hamilton-Jacobi equations, 
Darbon \& Osher have recently developed an algorithm which performs numerically well 
in the case of such high dimensional inviscid Hamilton-Jacobi equations; see \cite{MR3543239}.
Darbon \& Osher's algorithm is based on results from compressed sensing 
and on the Hopf formulas for the Hamilton-Jacobi equations.
A general algorithm for (nonlinear) parabolic PDEs
based on the Feynman-Kac and Bismut-Elworthy-Li formula and a multi-level decomposition of
Picard iteration was developed in \cite{Eetal2017} and has been 
shown to be quite efficient on a number examples in finance and physics. 
The complexity of the algorithm is shown to be 
$O(d \varepsilon^{-4})$ for semilinear heat equations, 
where $d$ is the dimensionality of the problem and $\varepsilon$ is the required accuracy.

In recent years, a new class of techniques, called deep learning, have
emerged in machine learning and have proven to be very effective in
dealing with a large class of high dimensional problems in computer vision (cf., e.g., \cite{Krizhevsky2012}),
natural language processing (cf., e.g., \cite{Hintonetal2012}), 
time series analysis, etc.\ (cf., e.g., \cite{Goodfellow2016,LeCun2015}).
This success fuels in speculations that deep learning might hold
the key to solve the curse of dimensionality problem.
It should be emphasized that at the present time, there are no theoretical
results that support such claims although the practical success of
deep learning has been astonishing.
However, this should not prevent us from trying to apply deep learning to
other problems where the curse of dimensionality has been the issue.

In this paper, we explore the use of deep learning for solving general high dimensional PDEs.
To this end, it is necessary to formulate the PDEs as a learning problem.
Motivated by  ideas in  \cite{Jiequn2016} where deep learning-based
algorithms were developed for high dimensional stochastic control 
problems, we explore a connection between (nonlinear) parabolic PDEs and 
backward stochastic differential equations (BSDEs) 
(see \cite{PardouxPeng1990,Peng1991,PardouxPeng1992})
since BSDEs share a lot of common features with stochastic control problems.

\section{Main ideas  of the algorithm}
\label{sec:derivation}

We will consider a fairly general class  of nonlinear parabolic PDEs 
(see \eqref{eq:PDE_numerics} in Subsection~\ref{sec:example_setting} below).
The proposed algorithm is based on the following set of ideas:
\begin{enumerate}[(i)]
\item Through the so-called nonlinear Feynman-Kac formula, we can formulate the
PDEs equivalently as BSDEs.
\item One can view the BSDE as a stochastic control problem
with the gradient of the solution being the policy function.
These stochastic control problems can then be viewed as model-based
reinforcement learning problems.
\item The (high dimensional) policy function can then be approximated by 
a deep neural network, as has been done in deep reinforcement learning.
\end{enumerate}

Instead of formulating initial value problems, as is
commonly done in the PDE literature, we 
consider the set up with terminal conditions since this facilitates
making connections with BSDEs. 
Terminal value problems can obviously be transformed to 
initial value problems and vice versa.

In the remainder of this section we present a rough sketch of the derivation of the proposed algorithm, which we refer to as
\emph{deep BSDE solver}. In this derivation we restrict ourself to a specific class of nonlinear PDEs, that is, we 
restrict ourself to semilinear heat equations (see \eqref{eq:PDE} below) and refer to 
Subsections~\ref{sec:general_case} and \ref{sec:example_setting} below for the general introduction of 
the deep BSDE solver.

\subsection{An example: a semilinear heat partial differential equation (PDE)}
\label{sec:the_PDE}

Let $ T \in (0,\infty) $, $ d \in \N $,
$ \xi \in \R^d $,
let
$ 
  f \colon 
  \R \times \R^d
  \to \R
$
and
$ 
  g \colon \R^d \to \R 
$
be continuous functions,
and let 
$ 
  u = ( u(t,x) )_{
    t \in [0,T], x \in \R^d
  } 
  \in C^{ 1, 2 }( [0,T] \times \R^d , \R ) 
$
satisfy 
for all $ t \in [0,T] $, $ x \in \R^d $ that
$ u(T,x) = g(x) $
and
\begin{equation}
\label{eq:PDE}
\tag{PDE}
\begin{split}
&
  \frac{ \partial  u}{ \partial t } ( t, x )
  +
  \frac{ 1 }{ 2 }
   ( \Delta_x u )( t, x )
  +
  f\big( 
    u(t,x), 
     ( \nabla_x u )( t, x ) 
  \big)
  = 
  0 .
\end{split}
\end{equation}
A key idea of this work is to reformulate the PDE~\eqref{eq:PDE}
as an appropriate \emph{stochastic control problem}.

\subsection{Formulation of the PDE as a suitable stochastic control problem}

More specifically, let
$ ( \Omega, \mathcal{F}, \P ) $ 
be a probability space,
let $ W \colon [0,T] \times \Omega \to \R^d $
be a $ d $-dimensional standard
Brownian motion on $ ( \Omega, \mathcal{F}, \P ) $,
let $ \mathbb{F} = ( \mathbb{F}_t )_{ t \in [0,T] } $
be the normal filtration on 
$ ( \Omega, \mathcal{F}, \P ) $
generated by $ W $, 
let $ \mathcal{A} $
be the set of all $ \mathbb{F} $-adapted $ \R^d $-valued 
stochastic processes with continuous sample paths,
and for every 
$ y \in \R $
and every $ Z \in \mathcal{A} $
let 
$ Y^{ y, Z } \colon [0,T] \times \Omega \to \R $ 
be an $ \mathbb{F} $-adapted stochastic process with continuous sample paths 
which satisfies that for all $ t \in [0,T] $
it holds $ \P $-a.s.\ that
\begin{equation}
\label{eq:FSDE_0}
  Y^{ y, Z }_t
  =
  y
  -
  \int_0^t 
  f\big( 
    Y^{ y, Z }_s, Z_s 
  \big)
  \, ds
  +
  \int_0^t
  \left<
    Z_s 
    , 
    dW_s
  \right>_{ \R^d }
  .
\end{equation}
We now view the solution 
$ 
  u \in 
  C^{ 1,2 }( [0,T] \times \R^d , \R ) 
$ 
of \eqref{eq:PDE} 
and its spatial derivative 
as the solution of a stochastic control problem 
associated to \eqref{eq:FSDE_0}.
More formally, 
under suitable regularity hypotheses on the nonlinearity $ f $
it holds that 
the pair consisting of
$ 
  u(0,\xi) \in \R 
$
and 
$
  (
    (
      \nabla_x u 
    )(t,\xi + W_t ) 
  )_{ t \in [0,T] }
  \in \mathcal{A}
$
is the (up to indistinguishability) unique global minimum of the function
\begin{equation}
\label{eq:stochastic_control2}
  \R \times \mathcal{A} 
  \ni
  ( y, Z )
  \mapsto
  \E\!\big[ 
    | 
      Y^{ y, Z }_T
      - 
      g( \xi + W_T ) 
    |^2
  \big]
  \in
  [0,\infty]
  .
\end{equation}
One can also view the stochastic control problem~\eqref{eq:FSDE_0}--\eqref{eq:stochastic_control2} 
(with $Z$ being the control)
as a model-based reinforcement learning problem.
In that analogy, we view 
$Z $
as the policy and we approximate $ Z \in \mathcal{A} $ using
feedforward neural networks (see \eqref{eq:gradient_approx} 
and Section~\ref{sec:numerics} below for further details). 
The process 
$ u( t, \xi + W_t ) $,
$ t \in [0,T] $,
corresponds to 
the value function associated to the stochastic
control problem and can be computed approximatively
by employing the policy $ Z $
(see \eqref{eq:Y_dynamics} below for details). 
The connection between the PDE~\eqref{eq:PDE}
and the stochastic control problem~\eqref{eq:FSDE_0}--\eqref{eq:stochastic_control2} 
is based on the nonlinear Feynman-Kac formula which 
links PDEs and BSDEs (see \eqref{eq:BSDE} and \eqref{eq:nonlinear_Feynman_Kac} below).

\subsection{The nonlinear Feynman-Kac formula}

Let 
$ Y \colon [0,T] \times \Omega \to \R $
and 
$ Z \colon [0,T] \times \Omega \to \R^d $
be $ \mathbb{F} $-adapted stochastic processes with continuous sample paths
which satisfy that for all 
$ t \in [0,T] $
it holds $ \P $-a.s.\ that
\begin{equation}
\tag{BSDE}
\label{eq:BSDE}
  Y_t 
  = g( \xi + W_T ) 
  + \int_t^T f( Y_s, Z_s ) \, ds
  - \int_t^T \left< Z_s , dW_s \right>_{ \R^d }
  .
\end{equation}
Under suitable additional regularity assumptions 
on the nonlinearity $ f $
we have that the nonlinear parabolic PDE~\eqref{eq:PDE} is 
related to the BSDE~\eqref{eq:BSDE}
in the sense that
for all $ t \in [0,T] $ 
it holds $ \P $-a.s.\ that
\begin{equation}
\label{eq:nonlinear_Feynman_Kac}
  Y_t 
  =
  u( t, \xi + W_t )
  \in \R
\qquad  
  \text{and}
\qquad 
  Z_t 
  =
  ( \nabla_x u )( t, \xi + W_t )
  \in \R^d
\end{equation}
(cf., e.g., \cite[Section~3]{PardouxPeng1992} and \cite{PardouxTang1999}). 
The first identity in~\eqref{eq:nonlinear_Feynman_Kac} 
is sometimes referred to as 
\emph{nonlinear Feynman-Kac formula} in the 
literature.

\subsection{Forward discretization of the backward stochastic differential equation (BSDE)}

To derive the deep BSDE solver, we first plug
the second identity in \eqref{eq:nonlinear_Feynman_Kac}
into \eqref{eq:BSDE}
to obtain that for all $ t \in [0,T] $
it holds $ \P $-a.s.\ that
\begin{equation}
\begin{split}
  Y_t 
&
  = g( \xi + W_T ) 
  + 
  \int_t^T 
  f\big( 
    Y_s, 
      ( \nabla_x u )( s, \xi + W_s ) 
  \big) \, ds
  - 
  \int_t^T 
    \left<
      ( \nabla_x u )( s, \xi + W_s ) 
      ,
      dW_s 
    \right>_{ \R^d }
  .
\end{split}
\end{equation}
In particular, we obtain that for all $ t_1, t_2 \in [0,T] $
with $ t_1 \leq t_2 $ it holds $ \P $-a.s.\ that
\begin{equation}
\label{eq:BSDE3}
\begin{split}
  Y_{ t_2 }
& =
  Y_{ t_1 }
  -
  \int_{ t_1 }^{ t_2 } 
  f\big( 
    Y_s, 
      ( \nabla_x u )( s, \xi + W_s ) 
  \big) \, ds
  +
  \int_{ t_1 }^{ t_2 } 
    \left<
      ( \nabla_x u )( s, \xi + W_s ) 
      ,
      dW_s 
    \right>_{ \R^d }
  .
\end{split}
\end{equation}
Next we apply 
a time discretization to 
\eqref{eq:BSDE3}.
More specifically, 
let $ N \in \N $ 
and let $ t_0, t_1, \dots, t_N \in [0,T] $
be real numbers which satisfy
\begin{equation}
  0 = t_0 < t_1 < \ldots < t_N = T
\end{equation}
and observe that 
\eqref{eq:BSDE3} 
suggests 
for $ N \in \N $ sufficiently large
that 
\begin{align}
\label{eq:approx_Y}
&
  Y_{ t_{ n + 1 } } 
\\
\nonumber
& \approx
  Y_{ t_n }
  -
  f\big( 
    Y_{ t_n } , 
    ( \nabla_x u )( t_n, \xi + W_{ t_n } ) 
  \big) 
  \,
  ( t_{ n + 1 } - t_n )
  +
  \left<
    ( \nabla_x u )( t_n, \xi + W_{ t_n } ) 
    ,
    W_{ t_{ n + 1 } }
    -
    W_{ t_n }
  \right>_{ \R^d }
  .
\end{align}

\subsection{Deep learning-based approximations}
In the next step we employ 
a deep learning approximation for 
\begin{equation}
\label{eq:deep_learning_policy}
  ( \nabla_x u )( t_n, x ) 
  \in \R^d
  ,
\qquad
  x \in \R^d ,
\qquad
  n \in \{ 0, 1, \dots, N \} ,
\end{equation}
but not for 
$ u( t_n, x) \in \R $,
$ x \in \R^d $,
$ n \in \{ 0, 1, \dots, N \} $.
Approximations for
$ u( t_n, x) \in \R $,
$ x \in \R^d $,
$ n \in \{ 0, 1, \dots, N \} $,
in turn,
can be computed recursively by using
\eqref{eq:approx_Y}
together with 
deep learning approximations for \eqref{eq:deep_learning_policy}.
More specifically, 
let 
$ 
  \rho \in \N 
$,
let 
$
  \mathcal{U}^{ \theta } 
  \in \R
$,
$ \theta \in \R^{ \rho } $,
be real numbers,
let
$
  \mathcal{V}^{ \theta }_n \colon
  \R^d \to \R^d
$,
$ n \in \{ 0, 1, \dots, N - 1 \} $,
$ \theta \in \R^{ \rho } $,
be continuous functions,
and let
$ 
  \mathcal{Y}^{ \theta } \colon 
  \{ 0, 1, \dots, N \} \times \Omega 
  \to \R
$,
$
  \theta \in \R^{ \rho }
$, 
be stochastic processes
which satisfy
for all 
$ \theta \in \R^{ \rho } $,
$ n \in \{ 0, 1, \dots, N - 1 \} $
that
$
  \mathcal{Y}^{ \theta }_0 
  = \mathcal{U}^{ \theta }
$
and
\begin{equation}
\label{eq:Y_dynamics}
\begin{split}
&
  \mathcal{Y}^{ \theta }_{ n + 1 }
  =
  \mathcal{Y}^{ \theta }_n
  -
  f\big(
    \mathcal{Y}^{ \theta }_n ,
    \mathcal{V}^{ \theta }_n( \xi + W_{ t_n } ) 
  \big)
  \,
  ( t_{ n + 1 } - t_n )
  +
  \left<
    \mathcal{V}^{ \theta }_n( \xi + W_{ t_n } ) 
    ,
      W_{ t_{ n + 1 } }
      -
      W_{ t_n }
  \right>_{ \R^d }
  .
\end{split}
\end{equation}
We think of $ \rho \in \N $ 
as the number of parameters in the neural network, 
for all appropriate
$ \theta \in \R^{ \rho } $
we think of 
$
  \mathcal{U}^{ \theta } 
  \in \R
$
as suitable approximations
\begin{equation}
\label{eq:solution_approx}
  \mathcal{U}^{ \theta }
  \approx
  u( 0 , \xi )
\end{equation}
of $ u( 0, \xi ) $,
and 
for all appropriate 
$ \theta \in \R^{ \rho } $,
$ x \in \R^d $,
$ n \in \{ 0, 1, \dots, N - 1 \} $
we think of
$
  \mathcal{V}^{ \theta }_n( x )
  \in \R^{ 1 \times d }
$
as suitable approximations
\begin{equation}
\label{eq:gradient_approx}
  \mathcal{V}^{ \theta }_n( x )
  \approx
  ( \nabla_x u )( t_n , x )
\end{equation}
of 
$
  ( \nabla_x u )( t_n , x )
$.
\subsection{Stochastic optimization algorithms}
The ``appropriate'' $ \theta \in \R^{ \rho } $ 
can be obtained by minimizing the expected loss function 
through stochastic gradient descent-type algorithms.
For the loss function we pick the squared approximation error 
associated to the terminal condition 
of the BSDE~\eqref{eq:BSDE}.
More precisely,
assume that the function 
$
  \R^{ \rho }
  \ni 
  \theta 
  \mapsto
  \E\big[ 
    |
      \mathcal{Y}^{ \theta }_N
      -
      g( \mathcal{X}_N )
    |^2
  \big]
  \in [0,\infty]
$
has a unique global minimum and 
let 
$ 
  \Lambda \in \R^{ \rho } 
$
be the real vector
for which the function
\begin{equation}
\label{eq:to_minimize}
  \R^{ \rho }
  \ni 
  \theta 
  \mapsto
  \E\big[ 
    |
      \mathcal{Y}^{ \theta }_N
      -
      g( \mathcal{X}_N )
    |^2
  \big]
  \in [0,\infty]
\end{equation}
is minimal. 
Minimizing the function \eqref{eq:to_minimize} 
is inspired by the fact that
\begin{equation}
  \E\!\left[ 
    | Y_T - g( X_T ) 
    |^2
  \right]
  = 0
\end{equation}
according to \eqref{eq:BSDE} above (cf.\ \eqref{eq:stochastic_control2} above).
Under suitable regularity assumptions,
we approximate 
the vector 
$ 
  \Lambda \in \R^{ \rho } 
$ 
through stochastic gradient descent-type 
approximation methods and thereby
we obtain random approximations 
$ 
  \Theta_0 ,
  \Theta_1 ,
  \Theta_2 ,
  \dots
  \colon \Omega 
  \to \R^{ \rho } 
$
of 
$ 
  \Lambda \in \R^{ \rho_{ \vartheta } } 
$. 
For sufficiently large $ N, \rho, m \in \N $
we then employ the random variable
$
  \mathcal{U}^{ \Theta_m }
  \colon \Omega \to \R
$
as a suitable implementable approximation 
\begin{equation}
  \mathcal{U}^{ \Theta_m }
  \approx
  u(0,\xi)
\end{equation}
of $ u(0,\xi) $ 
(cf.\ \eqref{eq:solution_approx} above)
and 
for sufficiently large $ N, \rho, m \in \N $ 
and all 
$ x \in \R^d $,
$ n \in \{ 0, 1, \dots, N - 1 \} $
we use the random variable 
$
  \mathcal{V}^{ \Theta_m }_n( x )
  \colon 
  \Omega \to \R^{ 1 \times d }
$
as a suitable implementable approximation 
\begin{equation}
  \mathcal{V}^{ \Theta_m }_n( x )
  \approx
  ( \nabla_x u )( t_n , x )
\end{equation}
of 
$ 
  ( \nabla_x u )( t_n , x )
$ (cf.\ \eqref{eq:gradient_approx} above).
In the next section the proposed approximation method 
is described in more detail.

To simplify the presentation we have restricted us in 
\eqref{eq:PDE}, \eqref{eq:FSDE_0}, \eqref{eq:stochastic_control2}, \eqref{eq:BSDE} above 
and Subsection~\ref{sec:specific} below 
to semilinear heat equations. We refer to Subsection~\ref{sec:general_case}
and Section~\ref{sec:numerics} below for the general description of the deep BSDE solver.

\section{Details of the algorithm}
\label{sec:algorithm}

\subsection{Formulation of the proposed algorithm in the case of semilinear heat equations}
\label{sec:specific}

In this subsection we describe the algorithm proposed
in this article in the specific situation 
where \eqref{eq:PDE} is the PDE under consideration,
where \emph{batch normalization} (see Ioffe \& Szegedy~\cite{Ioffe2015}) is not employed, 
and 
where the plain-vanilla stochastic gradient descent approximation method 
with a constant learning rate $ \gamma \in (0,\infty) $
and without mini-batches is the employed stochastic algorithm. 
The general framework, which includes the setting in this subsection 
as a special case, can be found in Subsection~\ref{sec:general_case}
below.

\begin{algo}[Specific case]
Let 
$ T, \gamma \in (0,\infty) $, $ d, \rho, N \in \N $,
$ \xi \in \R^d $, 
let 
$
  f \colon 
  \R \times \R^d
  \to \R
$
and
$
  g \colon \R^d \to \R
$
be functions,
let $ ( \Omega, \mathcal{F}, \P ) $ be a probability space, 
let 
$
  W^m \colon [0,T] \times \Omega \to \R^d
$, 
$ m \in \N_0 $,
be independent $ d $-dimensional standard Brownian motions on 
$
  ( \Omega, \mathcal{F}, \P )
$, 
let $ t_0, t_1, \dots, t_N \in [0,T] $
be real numbers with
\begin{equation}
  0 = t_0 < t_1 < \ldots < t_N = T
  ,
\end{equation}
for every 
$ \theta \in \R^{ \rho } $
let 
$ 
  \mathcal{U}^{ \theta } 
  \in \R
$,
for every 
$ \theta \in \R^{ \rho } $,
$ n \in \{ 0, 1, \ldots, N - 1 \} $
let
$ 
  \mathcal{V}_n^{  \theta } 
  \colon 
  \R^d
  \to \R^d
$
be a function,
for every 
$ m \in \N_0 $,
$
  \theta \in \R^{ \rho }  
$
let 
$ 
  \mathcal{Y}^{ \theta, m } 
  \colon \{ 0, 1, \ldots, N \} \times \Omega \to \R^k
$
be the stochastic process which satisfies for all 
$
  n \in \{ 0, 1, \ldots, N - 1 \}
$ 
that
$
  \mathcal{Y}_0^{ \theta, m } 
  = 
  \mathcal{U}^{ \theta }
$
and
\begin{equation}
  \mathcal{Y}_{ n + 1 }^{ \theta, m }
  = 
  \mathcal{Y}_n^{ \theta, m } 
  -
  f\big(
    \mathcal{Y}_n^{ \theta, m }
    , 
    \mathcal{V}_n^{ \theta }(
      \xi + W^m_{ t_n }
    )
  \big) 
  \,
  ( t_{ n + 1 } - t_n )
  + 
  \left<
    \mathcal{V}_n^{ \theta }(
      \xi + W^m_{ t_n }
    )
    ,
    W_{ t_{ n + 1 } }^{ m }
    - 
    W_{ t_n }^{ m }
  \right>_{ \R^d }
  ,
\end{equation}
for every 
$ m \in \N_0 $
let
$
  \phi^m
  \colon
  \R^{ \rho } 
  \times 
  \Omega
  \to 
  \R
$
be the function which satisfies for all 
$ \theta \in \R^{ \rho } $,
$ \omega \in \Omega $
that
\begin{equation}
\begin{split}
  \phi^m( \theta, \omega ) 
  = 
  \big| 
    \mathcal{Y}_N^{ \theta, m }( \omega )
    -
    g(
      \xi + W^m_T( \omega )
    ) 
  \big|^2
  ,
\end{split}
\end{equation}
for every 
$ m \in \N_0 $
let
$
  \Phi^m
  \colon
  \R^{ \rho } 
  \times 
  \Omega
  \to 
  \R^{ \rho }
$
be a function which satisfies 
for all
$ \omega \in \Omega $,
$ 
  \theta \in 
  \{ 
    v \in 
    \R^{ \rho } 
    \colon
    (
      \R^{ \rho } 
      \ni w \mapsto 
      \phi^{ m }_{ {\bf s} }( w, \omega )
      \in
      \R
      \text{ is differentiable at }
      v \in \R^{ \rho }
    )
  \}
$
that
\begin{equation}
  \Phi^m( \theta, \omega ) 
  =
  ( \nabla_{ \theta } \phi^m )( \theta, \omega ) 
  ,
\end{equation}
and
let 
$ 
  \Theta \colon \N_0 \times \Omega \to 
  \R^{ \rho }
$ 
be a stochastic process
which satisfy for all $ m \in \N $ that 
\begin{equation}
  \Theta_m 
  =
  \Theta_{ m - 1 } 
  -
  \gamma \cdot
  \Phi^m(
    \Theta_{ m - 1 } 
  )
  .
\end{equation}
\end{algo}

Under suitable further hypotheses (cf.\ Sections~\ref{sec:numerics} and \ref{sec:special_cases} below),
we think 
in the case of sufficiently large
$ \rho, N, m \in \N $ 
and sufficiently small $ \gamma \in (0,\infty) $
of
$
  \mathcal{U}^{ \Theta_m } \in \R
$
as an appropriate approximation 
\begin{equation}
  u(0,\xi) \approx \mathcal{U}^{ \Theta_m }
\end{equation}
of the solution 
$ u(t,x) \in \R $, 
$ (t,x) \in [0,T] \times \R^d $,
of the PDE
\begin{equation}
  \frac{ \partial u}{ \partial t } ( t, x )
  +
  \frac{ 1 }{ 2 } 
  \,
  ( \Delta_x u )( t, x )
  +
  f\big( 
    u(t,x),  ( \nabla_x u )(t,x) 
  \big)
  = 0
\end{equation}
for 
$ (t,x) \in [0,T] \times \R^d $.

\subsection{Formulation of the proposed algorithm in the general case}
\label{sec:general_case}

\begin{algo}[General case]
Let 
$ T \in (0,\infty) $, $ d, k, \rho, \varrho, N, \varsigma \in \N $,
$ \xi \in \R^d $, 
let 
$
  f \colon [0,T] \times \R^d \times \R^k \times \R^{ k \times d }
  \to \R
$,
$
  g \colon \R^d \to \R^k
$,
and
$
  \Upsilon
  \colon 
    [0,T]^2 \times \R^d  
    \times \R^d
  \to
    \R^d
$
be functions,
let $ ( \Omega, \mathcal{F}, \P ) $ be a probability space, 
let 
$
  W^{ m, j } \colon [0,T] \times \Omega \to \R^d
$, 
$ m, j \in \N_0 $,
be independent $ d $-dimensional standard Brownian motions on 
$
  ( \Omega, \mathcal{F}, \P )
$, 
let $ t_0, t_1, \dots, t_N \in [0,T] $
be real numbers with
\begin{equation}
  0 = t_0 < t_1 < \ldots < t_N = T
  ,
\end{equation}
for every 
$ \theta \in \R^{ \rho } $
let 
$ 
  \mathcal{U}^{ \theta } 
  \in \R^k
$,
for every 
$ \theta \in \R^{ \rho } $,
$
  {\bf s} \in \R^{ \varsigma } 
$,
$ n \in \{ 0, 1, \ldots, N - 1 \} $,
$ j \in \N_0 $
let
$ 
  \mathcal{V}_{ n, j }^{  \theta, {\bf s} } 
  \colon 
  ( \R^d )^{ \N }
  \to \R^{ k \times d }
$
be a function,
for every 
$ m, j \in \N_0 $
let 
$
  \mathcal{X}^{ m, j } 
  \colon \{ 0, 1, \dots, N \} \times \Omega \to \R^d
$
and
$ 
  \mathcal{Y}^{ \theta, {\bf s}, m, j } 
  \colon \{ 0, 1, \ldots, N \} \times \Omega \to \R^k
$,
$
  \theta \in \R^{ \rho }  
$,
$
  {\bf s} \in \R^{ \varsigma }
$,
be stochastic processes which satisfy for all 
$
  \theta \in \R^{ \rho }
$,
$
  {\bf s} \in \R^{ \varsigma }
$,
$
  n \in \{ 0, 1, \ldots, N - 1 \}
$ 
that
\begin{equation}
  \mathcal{X}^{ m, j }_0 = \xi
  ,
  \qquad
  \mathcal{Y}_0^{ \theta, {\bf s}, m, j } 
  = 
  \mathcal{U}^{ \theta }
  ,
  \qquad
\label{eq:X_dynamic}
  \mathcal{X}^{ m, j }_{ n + 1 }
  =
  \Upsilon\big(   
    t_n 
    ,
    t_{ n + 1 } 
    ,
    \mathcal{X}^{ m, j }_n
    ,
    W^{ m, j }_{ t_{ n + 1 } } - W^{ m, j }_{ t_n }
  \big)
  ,
\end{equation}
\begin{multline}
  \mathcal{Y}_{ n + 1 }^{ \theta, {\bf s}, m, j }
  = 
  \mathcal{Y}_n^{ \theta, {\bf s}, m, j } 
  -
  f\big(
    t_n
    ,
    \mathcal{X}_n^{ m, j }
    , 
    \mathcal{Y}_n^{ \theta, {\bf s}, m, j }
    , 
    \mathcal{V}_{ n, j }^{ \theta, {\bf s} }(
      \{
        \mathcal{X}_n^{ m, i }
      \}_{ 
        i \in \N
      }
    )
  \big) 
  \,
  ( t_{ n + 1 } - t_n )
\\
  + 
    \mathcal{V}_{ n, j }^{ \theta, {\bf s} }(
      \{ 
        \mathcal{X}_n^{ m, i }
      \}_{ 
        i \in \N
      }
    )
    \,
  (
    W_{ t_{ n + 1 } }^{ m, j }
    - 
    W_{ t_n }^{ m, j }
  )
  ,
\end{multline}
for every 
$ m, j \in \N_0 $,
$ {\bf s} \in \R^{ \varsigma } $
let
$
  \phi^{ m, j }_{ {\bf s} }
  \colon
  \R^{ \rho } 
  \times 
  \Omega
  \to 
  \R
$
be the function which satisfies for all 
$ \theta \in \R^{ \rho } $,
$ \omega \in \Omega $
that
\begin{equation}
\begin{split}
  \phi^{ m, j }_{ {\bf s} }( \theta, \omega ) 
  = 
  \| 
    \mathcal{Y}_N^{ \theta, {\bf s}, m, j }( \omega )
    -
    g(
      \mathcal{X}_N^{ m, j }( \omega )
    ) 
  \|_{ \R^k }^2
  ,
\end{split}
\end{equation}
for every 
$ m, j \in \N_0 $,
$ {\bf s} \in \R^{ \varsigma } $
let
$
  \Phi^{ m, j }_{ {\bf s} }
  \colon
  \R^{ \rho } 
  \times 
  \Omega
  \to 
  \R^{ \rho }
$
be a function which satisfies 
for all
$ \omega \in \Omega $,
$ 
  \theta \in 
  \{ 
    v \in 
    \R^{ \rho } 
    \colon
    (
      \R^{ \rho } 
      \ni w \mapsto 
      \phi^{ m, j }_{ {\bf s} }( w, \omega )
      \in
      \R
      \text{ is differentiable at }
      v \in \R^{ \rho }
    )
  \}
$
that
\begin{equation}
  \Phi_{ {\bf s} }^{ m, j }( \theta, \omega ) 
  =
  ( \nabla_{ \theta } \phi_{ {\bf s} }^{ m, j } )( \theta, \omega ) 
  ,
\end{equation}
let 
$
  \mathcal{S} \colon 
  \R^{ \varsigma } \times 
  \R^{ \rho } \times 
  ( \R^d )^{ \{ 0, 1, \dots, N - 1 \} \times \N }
  \to \R^{ \varsigma }
$
be a function,
for every $ m \in \N $
let
$
  \psi_m 
  \colon 
  \R^{ \varrho }
  \to
  \R^{ \rho }
$
and
$
  \Psi_m 
  \colon
  \R^{ \varrho }
  \times
  (
    \R^{ \rho }
  )^{ \N }
  \to 
  \R^{ \varrho }
$
be functions,
and
let 
$
  \mathbb{S} \colon \N_0 \times \Omega \to \R^{ \varsigma }
$,
$ 
  \Xi \colon \N_0 \times \Omega \to 
  \R^{ \varrho }
$, 
and
$ 
  \Theta \colon \N_0 \times \Omega \to 
  \R^{ \rho }
$ 
be stochastic processes 
which satisfy for all $ m \in \N $ that 
\begin{equation}
\label{eq:S_dynamics}
  \mathbb{S}_m
  =
  \mathcal{S}\big(
    \mathbb{S}_{ m - 1 }
    ,
    \Theta_{ m - 1 }
    ,
    \{ 
      \mathcal{X}^{ m - 1, i }_n
    \}_{ 
      (n, i) \in 
      \{ 0, 1, \dots, N - 1 \} \times \N
    }
  \big)
  ,
\end{equation}
\begin{equation}
\label{eq:theta_dynamics}
  \Xi_{ m }
  = 
  \Psi_m\big(
    \Xi_{ m - 1 }
    ,
    \{
      \Phi^{ m - 1, j }_{ \mathbb{S}_m }( 
        \Theta_{ m - 1 }
      )
    \}_{ j \in \N }
  \big)
  ,
%
\qquad
  \text{and}
\qquad 
  \Theta_m 
  =
  \Theta_{ m - 1 } 
  -
  \psi_m( \Xi_m )
  .
\end{equation}
\end{algo}

\subsection{Comments on the proposed algorithm}

The dynamics in \eqref{eq:X_dynamic} associated 
to the stochastic processes 
$
  ( \mathcal{X}^{ m, j }_n )_{ n \in \{ 0, 1, \dots, N \} }
$
for 
$
  m, j \in \N_0
$
allows us to incorporate 
different algorithms 
for the discretization of the considered 
forward stochastic differential equation (SDE) into the deep BSDE solver in Subsection~\ref{sec:general_case}. 
The dynamics in \eqref{eq:theta_dynamics}
associated to the stochastic processes 
$ \Xi_m $, $ m \in \N_0 $,
and 
$ \Theta_m $, $ m \in \N_0 $,
allows us to incorporate different 
stochastic approximation algorithms 
such as 
\begin{itemize}
\item 
stochastic gradient descent with or without mini-batches 
(see Subsection~\ref{sec:SGD} below)
as well as 
\item
adaptive moment estimation (Adam) 
with mini-batches
(see Kingma \& Jimmy~\cite{Kingma2015} and Subsection~\ref{sec:Adam} below)
into the deep BSDE solver in Subsection~\ref{sec:general_case}.
\end{itemize}
The dynamics in \eqref{eq:S_dynamics} 
associated to the stochastic process 
$ \mathbb{S}_m $, $ m \in \N_0 $,
allows us to incorporate the standardization procedure in 
\emph{batch normalization}  
(see Ioffe \& Szegedy~\cite{Ioffe2015} and also 
Section~\ref{sec:numerics} below) into 
the deep BSDE solver in Subsection~\ref{sec:general_case}. In that case we think of
$ \mathbb{S}_m $, $ m \in \N_0 $, 
as approximatively calculated means and standard deviations.

\section{Examples for nonlinear partial differential equations (PDEs)
and nonlinear backward stochastic differential equations (BSDEs)}
\label{sec:numerics}

In this section we illustrate the algorithm
proposed in Subsection~\ref{sec:general_case} 
using several concrete example PDEs. 
In the examples below we will employ 
the general approximation method 
in Subsection~\ref{sec:general_case} 
in conjunction with the Adam optimizer (cf.\ Example~\ref{ex:Adam} below and Kingma \& Ba~\cite{Kingma2015})
with mini-batches with $ 64 $ samples in each iteration step 
(see Subsection~\ref{sec:example_setting} for a detailed description).

In our implementation we employ
$ N - 1 $ 
fully-connected feedforward neural networks to represent 
$
  \mathcal{V}^{ \theta }_{ n, j } 
$
for $ n \in \{ 1, 2, \dots, N - 1 \} $,
$ j \in \{ 1, 2, \dots, 64 \} $,
$ \theta \in \R^{ \rho } $ (cf.\ also Figure~\ref{fig:nn_architecture}
below for a rough sketch of the architecture of the deep BSDE solver). 
Each of the neural networks consists of $ 4 $ layers
($ 1 $ input layer [$ d $-dimensional], $ 2 $ hidden layers [both $ d+10 $-dimensional], 
and $ 1 $ output layer [$ d $-dimensional]).
The number of hidden units in each hidden layer is equal to $ d + 10 $.
We also adopt batch normalization (BN) (see Ioffe \& Szegedy~\cite{Ioffe2015}) 
right after each matrix multiplication and before activation. 
We employ the rectifier function
$
  \R \ni x \mapsto \max\{ 0, x \} \in [0,\infty) 
$
as our activation function for the hidden variables. 
All the weights in the network are initialized using a normal or a uniform distribution 
without any pre-training. 
Each of the numerical experiments presented below 
is performed in {\sc Python} using {\sc TensorFlow}
on a {\sc Macbook Pro} with a $ 2.90 $ Gigahertz (GHz) {\sc Intel Core} i5 
micro processor and 16 gigabytes (GB) of 1867 Megahertz (MHz)
double data rate type three synchronous dynamic 
random-access memory (DDR3-SDRAM).
We also refer to 
the {\sc Python} code~\ref{code:deepPDEmethod} in Subsection~\ref{sec:Python_code} below for an  
implementation of the deep BSDE solver in the case of the $ 100 $-dimensional Allen-Cahn PDE~\eqref{eq:PDE_allencahn}.

\subsection{Setting}
\label{sec:example_setting}

Assume the setting in Subsection~\ref{sec:general_case}, 
assume for all 
$ \theta = ( \theta_1, \dots, \theta_{ \rho } ) \in \R^{ \rho } $
that 
$ k = 1 $, 
$ 
  \rho = 
  d + 1 
  +
  ( N - 1 )
  \left(
    2 d ( d + 10 )
    +
    ( d + 10 )^2
    +
    4 ( d + 10 )
    + 2 d
  \right)
$,
$
  \varrho = 2 \rho 
$,
$
  \mathcal{U}^{ \theta } = \theta_1
$,
$
  \Xi_0 = 0
$,
let 
$ \mu \colon [0,T] \times \R^d \to \R^d $ 
and
$ \sigma \colon [0,T] \times \R^d \to \R^{ d \times d } $ 
be functions, 
let 
$
  u \colon [0,T] \times \R^d \to \R  
$ 
be a continuous and at most polynomially growing function which satisfies for all 
$
  (t,x) \in [0,T) \times \R^d
$
that
$
  u|_{ [0,T) \times \R^d } 
  \in
  C^{ 1, 2 }( [0,T) \times \R^d, \R ) 
$,
$
  u(T,x) = g(x) 
$, 
and
\begin{multline}  
\label{eq:PDE_numerics}
   \frac{ \partial u}{ \partial t } ( t, x )
  +
  \frac{ 1 }{ 2 }
  \operatorname{Trace}\!\big(
    \sigma( t, x ) 
    \, 
    [ \sigma( t, x ) ]^*
    \,
    ( \operatorname{Hess}_x u)( t, x )
  \big) 
  +
  \langle 
    \mu( t, x ) 
    ,
    ( \nabla_x u )( t, x ) 
  \rangle
\\
  + 
  f\big( 
    t, x, 
    u( t, x )    
    ,
    [ ( \nabla_x u)( t, x ) ]^{ * }
    \,
    \sigma( t, x )
  \big)
  = 0
  ,
\end{multline}
let $ \varepsilon = 10^{ - 8 } $,
$ \mathbb{X} = \frac{ 9 }{ 10 } $,
$ \mathbb{Y} = \frac{ 999 }{ 1000 } $,
$ J = 64 $,
$ ( \gamma_m )_{ m \in \N } \subseteq (0,\infty) $,
let 
$
  \operatorname{Pow}_r \colon
  \R^{ \rho }
  \to 
  \R^{ \rho }
$,
$ r \in (0,\infty) $,
be the functions 
which satisfy
for all 
$ r \in (0,\infty) $,
$ 
  x = ( x_1, \dots, x_{ \rho } ) \in \R^{ \rho }
$ 
that
\begin{equation}
  \operatorname{Pow}_{ r }( x )
  = ( | x_1 |^r, \dots, | x_{ \rho } |^r )
  ,
\end{equation}
and assume 
for all $ m \in \N $,
$ x, y \in \R^{ \rho } $,
$ 
  ( \varphi_j )_{ j \in \N } \in 
  ( \R^{ \rho } )^{ \N }
$
that
\begin{equation}
  \Psi_m( x, y, ( \varphi_j )_{ j \in \N } )
  =
  \big(
  \mathbb{X} x
  +
    ( 1 - \mathbb{X} ) 
    \big(
      \tfrac{
        1 
      }{ 
        J
      }
      \smallsum_{ j = 1 }^{ J }
      \varphi_j
    \big)
  ,
  \mathbb{Y} y
  +
    ( 1 - \mathbb{Y} ) 
    \operatorname{Pow}_2\!\big(
      \tfrac{ 1 }{ J }
      \smallsum_{ j = 1 }^{ J }
      \varphi_j
    \big)
  \big)
\end{equation}
and
\begin{equation}
  \psi_m( x, y ) 
  = 
  \left[ 
    \varepsilon
    +
    \operatorname{Pow}_{ \nicefrac{ 1 }{ 2 } }\!\left(
      \frac{
        y
      }{
        ( 1 - \mathbb{Y}^m )
      }
    \right)
  \right]^{ - 1 }
  \frac{ 
    \gamma_m x 
  }{ 
    ( 1 - \mathbb{X}^m ) 
  }
  .
\end{equation}
(cf.\ Example~\ref{ex:Adam} below and Kingma \& Ba~\cite{Kingma2015}).

\begin{remark}
In this remark we illustrate the 
specific choice of the dimension $ \rho \in \N $ 
of $ \theta = ( \theta_1, \dots, \theta_{ \rho } ) \in \R^{ \rho } $ 
in the framework in Subsection~\ref{sec:example_setting} above.
\begin{enumerate}[(i)]
\item 
\label{item:i}
The first component 
of $ \theta = ( \theta_1, \dots, \theta_{ \rho } ) \in \R^{ \rho } $
is employed for approximating the real number $ u( 0, \xi ) \in \R $.

\item
\label{item:ii}
The next $ d $-components
of $ \theta = ( \theta_1, \dots, \theta_{ \rho } ) \in \R^{ \rho } $
are employed for approximating the components of the 
$ ( 1 \times d ) $-matrix
$ 
  ( \tfrac{ \partial }{ \partial x } u )(0,\xi) \, \sigma(0,\xi)
  \in \R^{ 1 \times d }
$.

\item 
\label{item:LN1}
In each of the employed $ N - 1 $ neural networks 
we use $ d ( d + 10 ) $ components 
of $ \theta = ( \theta_1, \dots, \theta_{ \rho } ) \in \R^{ \rho } $
to describe the linear transformation from the $ d $-dimensional first layer (input layer)
to the $ (d + 10) $-dimensional second layer (first hidden layer)
(to uniquely describe a real $ (d+10) \times d $-matrix).

\item 
\label{item:LN2}
In each of the employed $ N - 1 $ neural networks 
we use $ ( d + 10 )^2 $ components 
of $ \theta = ( \theta_1, \dots, \theta_{ \rho } ) \in \R^{ \rho } $
to uniquely describe the linear transformation from the $ ( d + 10 ) $-dimensional second layer 
(first hidden layer) to the $ ( d + 10 ) $-dimensional third layer (second hidden layer)
(to uniquely describe a real $ (d + 10) \times ( d + 10 ) $-matrix).

\item 
\label{item:LN3}
In each of the employed $ N - 1 $ neural networks 
we use $ d ( d + 10 ) $ components 
of $ \theta = ( \theta_1, \dots, \theta_{ \rho } ) \in \R^{ \rho } $
to describe the linear transformation from the $ (d+10) $-dimensional third layer (second hidden layer)
to the $ d $-dimensional fourth layer (output layer)
(to uniquely describe a real $ d \times ( d + 10 ) $-matrix).

\item 
\label{item:last}
After each of the linear transformations in items~\eqref{item:LN1}--\eqref{item:LN3} above 
we employ a componentwise affine linear transformation (multiplication with a diagonal matrix 
and addition of a vector) within the batch normalization procedure, 
i.e.,
in each of the employed $ N - 1 $ neural networks, 
we use $ 2 ( d + 10 ) $ components 
of $ \theta = ( \theta_1, \dots, \theta_{ \rho } ) \in \R^{ \rho } $
for the componentwise affine linear transformation between 
the first linear transformation (see item~\eqref{item:LN1})
and the first application of the activation function,
we use $ 2 ( d + 10 ) $ components 
of $ \theta = ( \theta_1, \dots, \theta_{ \rho } ) \in \R^{ \rho } $
for the componentwise affine linear transformation between 
the second linear transformation (see item~\eqref{item:LN2})
and the second application of the activation function,
and 
we use $ 2 d $ components 
of $ \theta = ( \theta_1, \dots, \theta_{ \rho } ) \in \R^{ \rho } $
for the componentwise affine linear transformation after
the third linear transformation (see item~\eqref{item:LN3}).


\end{enumerate}
Summing \eqref{item:i}--\eqref{item:last} results in 
\begin{equation}
\begin{split}
  \rho 
& =
  \underbrace{
    1 
    +
    d
  }_{
    \text{items~\eqref{item:i}--\eqref{item:ii}}
  }
  +
  \underbrace{
    ( N - 1 )
    \left(
      d ( d + 10 )  
      +
      ( d + 10 )^2
      + 
      d ( d + 10 )  
    \right)
  }_{
    \text{items~\eqref{item:LN1}--\eqref{item:LN3}}
  }
\\ &
\quad
  +
  \underbrace{
    ( N - 1 )
    \left(
      2 ( d + 10 )
      +
      2 ( d + 10 )
      +
      2 d
    \right)
  }_{
    \text{item~\eqref{item:last}}
  }
\\
& 
=
  d + 1 
  +
  ( N - 1 )
  \left(
    2 d ( d + 10 )
    +
    ( d + 10 )^2
    +
    4 ( d + 10 )
    + 2 d
  \right)
  .
\end{split}
\end{equation}
\end{remark}

\begin{figure}[ht]
\includegraphics[width=15cm]{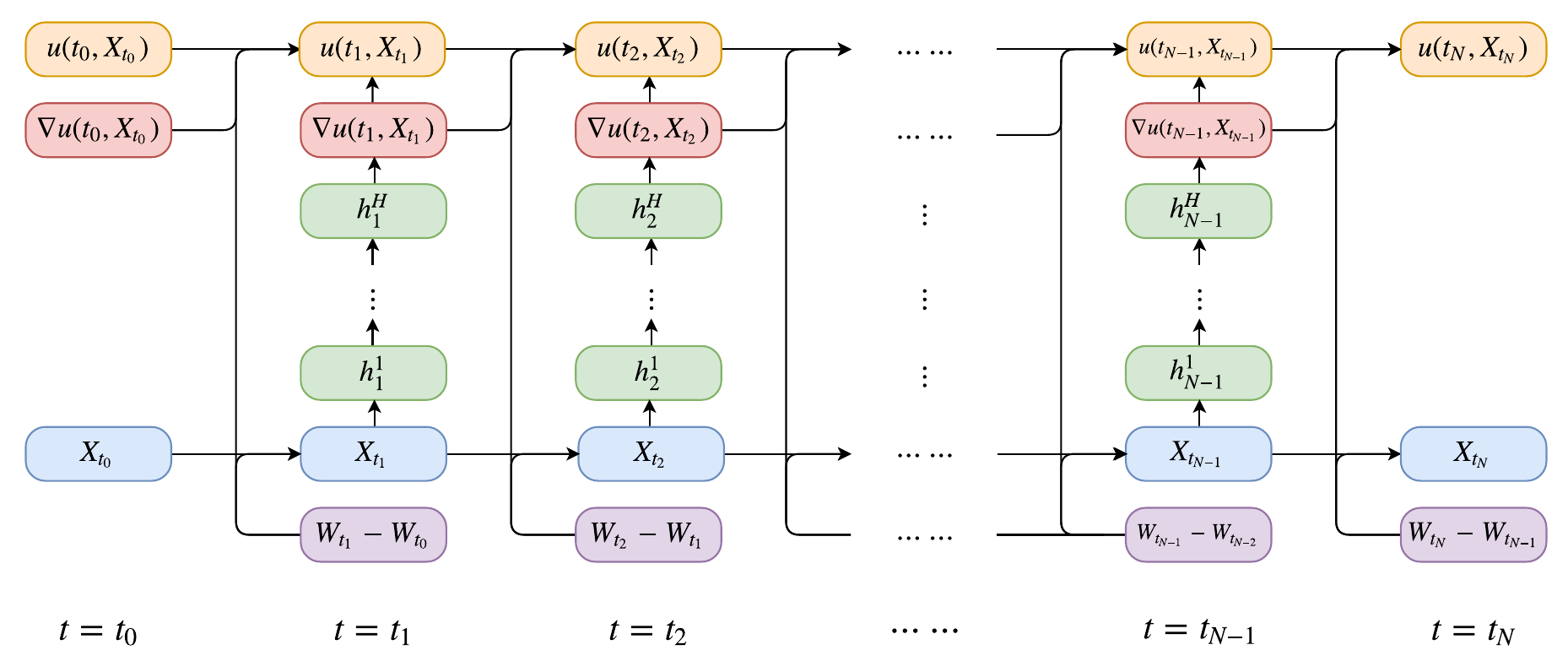}
\caption{Rough sketch of the architecture of the deep BSDE solver.\label{fig:nn_architecture}}
\end{figure}

\subsection{Allen-Cahn equation}
\label{subsec:allencahn}

In this section we test the deep BSDE solver in the case of an $ 100 $-dimensional Allen-Cahn PDE 
with a cubic nonlinearity (see \eqref{eq:PDE_allencahn} below).

More specifically, assume the setting in the Subsection~\ref{sec:example_setting} and
assume for all
$
  s, t \in [0,T]
$, 
$
  x,
  w \in \R^d
$,
$
  y \in \R
$, 
$
  z \in \R^{ 1 \times d }
$,
$ m \in \N $
that
$
  \gamma_m =
    5 \cdot 10^{ - 4 }
$,
$ d = 100 $,
$ T = \frac{ 3 }{ 10 } $,
$ N = 20 $,
$
  \mu(t,x) = 0
$,
$
  \sigma(t,x) w = \sqrt{2} \, w
$,
$
  \xi =  ( 0, 0, \dots, 0 ) \in \R^d
$,
$
  \Upsilon( s, t, x, w )
  =
  x + \sqrt{2} \, w
$,
$
  f(t,x,y,z) = y - y^3
$,
and
$
  g(x) 
  = 
  \left[
    2 + 
    \frac{ 2 }{ 5 } \, \| x \|_{ \R^d }^2
  \right]^{ - 1 }
$.
Note that the solution $ u $ 
of the PDE \eqref{eq:PDE_numerics} then satisfies for all 
$
  t \in [0,T)
$, 
$
  x \in \R^d
$
that
$
  u(T,x) = g(x) 
$
and
\begin{equation}  
\label{eq:PDE_allencahn}
\begin{split}
&
   \frac{ \partial u}{ \partial t } ( t, x )
  + 
  u(t,x) - \left[ u(t,x) \right]^3
  +
  ( \Delta_x u )(t,x)
  = 0
  .
\end{split}     
\end{equation}
In Table~\ref{tab:1} we approximatively 
calculate the mean 
of 
$ \mathcal{U}^{ \Theta_m } $, 
the standard deviation
of 
$ \mathcal{U}^{ \Theta_m } $, 
the relative $ L^1 $-approximatin error
associated to 
$ \mathcal{U}^{ \Theta_m } $,
the standard deviation of 
the relative $ L^1 $-approximatin error
associated to 
$ \mathcal{U}^{ \Theta_m } $, 
and the runtime 
in seconds needed to calculate 
one realization of $ \mathcal{U}^{ \Theta_m } $
against 
$ m \in \{ 0, 1000, 2000, 3000, 4000 \} $
based on $ 5 $ independent realizations
($ 5 $ independent runs)
(see also the {\sc Python} code~\ref{code:deepPDEmethod} below).
Table~\ref{tab:1} also depicts 
the mean of the loss function associated to 
$ \Theta_m $ 
and the standard deviation of the loss 
function associated to 
$ \Theta_m $ 
against 
$ m \in \{ 0, 1000, 2000, 3000, 4000 \} $
based on $ 256 $ Monte Carlo samples 
and $ 5 $ independent realizations 
($ 5 $ independent runs).
In addition, 
the relative $ L^1 $-approximation error 
associated to
$ \mathcal{U}^{ \Theta_m } $
against 
$ m \in \{ 1, 2, 3, \dots, 4000 \} $
is pictured on the left hand side 
of Figure~\ref{fig:allen_cahn}
based on $ 5 $ independent realizations ($ 5 $ independent runs)
and
the mean of the loss function 
associated to 
$ \Theta_m $ 
against 
$ m \in \{ 1, 2, 3, \dots, 4000 \} $
is pictured on the right hand side of 
Figure~\ref{fig:allen_cahn}
based on $ 256 $ Monte Carlo samples 
and $ 5 $ independent realizations ($ 5 $ independent runs).
In the approximative computations of the relative $ L^1 $-approximation errors 
in Table~\ref{tab:1} and Figure~\ref{fig:allen_cahn}
the value 
$
  u(0,\xi) = u(0,0,\dots, 0) 
$
of the exact solution $ u $
of the PDE~\eqref{eq:PDE_allencahn} is replaced by
the value $ 0.052802 $ which, in turn, 
is calculated by means 
of the Branching diffusion method 
(see the {\sc Matlab} code~\ref{code:Branching} below 
and see, e.g., 
\cite{Labordere2012,HenryLabordereTanTouzi2014,Labordereetal2016} 
for analytical and numerical results 
for the Branching diffusion method in the literature). 
\begin{center}
\begin{table}[!ht]
\begin{center}
\begin{tabular}{|c|c|c|c|c|c|c|c|}
\hline
Number 
&
Mean
&
Standard
&
Relative
&
Standard
&
Mean
&
Standard
&
Runtime
\\
of
&
of
$ \mathcal{U}^{ \Theta_m } $
&
deviation
&
$ L^1 $-appr.
&
deviation
&
of the 
&
deviation
&
in sec.
\\
iteration
&
&
of $ \mathcal{U}^{ \Theta_m } $
&
error
&
of the
&
loss
&
of the 
&
for one
\\
steps $ m $
&
&
&
&
relative
&
function
&
loss
&
realization
\\
&
&
&
&
$ L^1 $-appr.
&
&
function
&
of $ \mathcal{U}^{ \Theta_m } $
\\
&
&
&
&
error
&
&
&
\\
\hline
0
&
0.4740
&
0.0514
&
7.9775
&
0.9734
&
0.11630
&
0.02953
&

\\
\hline
1000
&
0.1446
&
0.0340
&
1.7384
&
0.6436
&
0.00550
&
0.00344
&
201
\\
\hline
2000
&
0.0598
&
0.0058
&
0.1318
&
0.1103
&
0.00029
&
0.00006
&
348
\\
\hline
3000
&
0.0530
&
0.0002
&
0.0050
&
0.0041
&
0.00023
&
0.00001
&
500
\\
\hline
4000
&
0.0528
&
0.0002
&
0.0030
&
0.0022
&
0.00020
&
0.00001
&
647
\\
\hline
\end{tabular}
\end{center}
\caption{Numerical simulations for the deep BSDE solver 
in Subsection~\ref{sec:general_case} in the case of the PDE~\eqref{eq:PDE_allencahn}. 
\label{tab:1}}
\end{table}
\end{center}
\begin{figure}[ht]
\centering
\setcounter{subfigure}{0}
\subfigure[Relative $ L^1 $-approximation error]{\includegraphics[width=8cm]{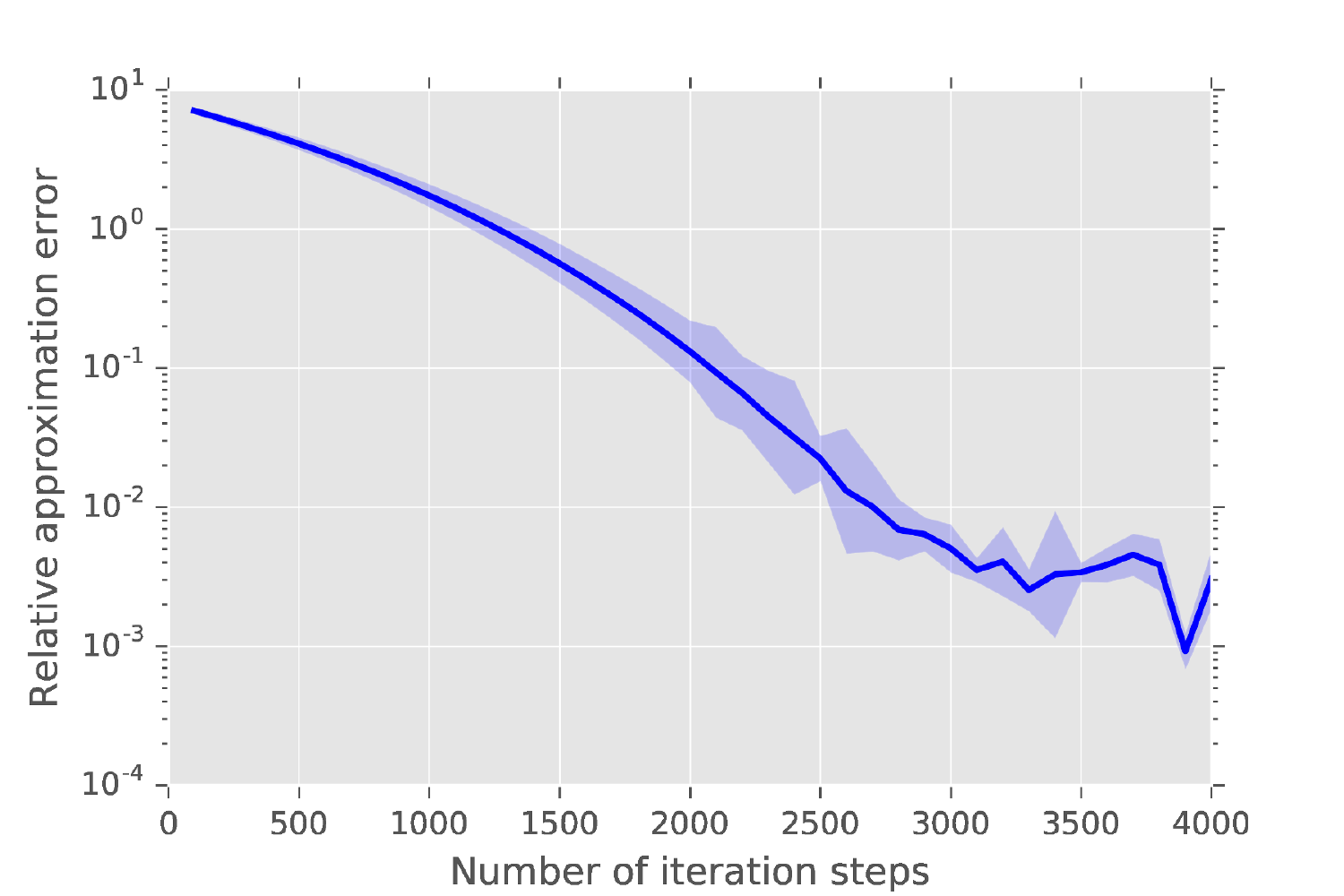}}
\subfigure[Mean of the loss function]{\includegraphics[width=8cm]{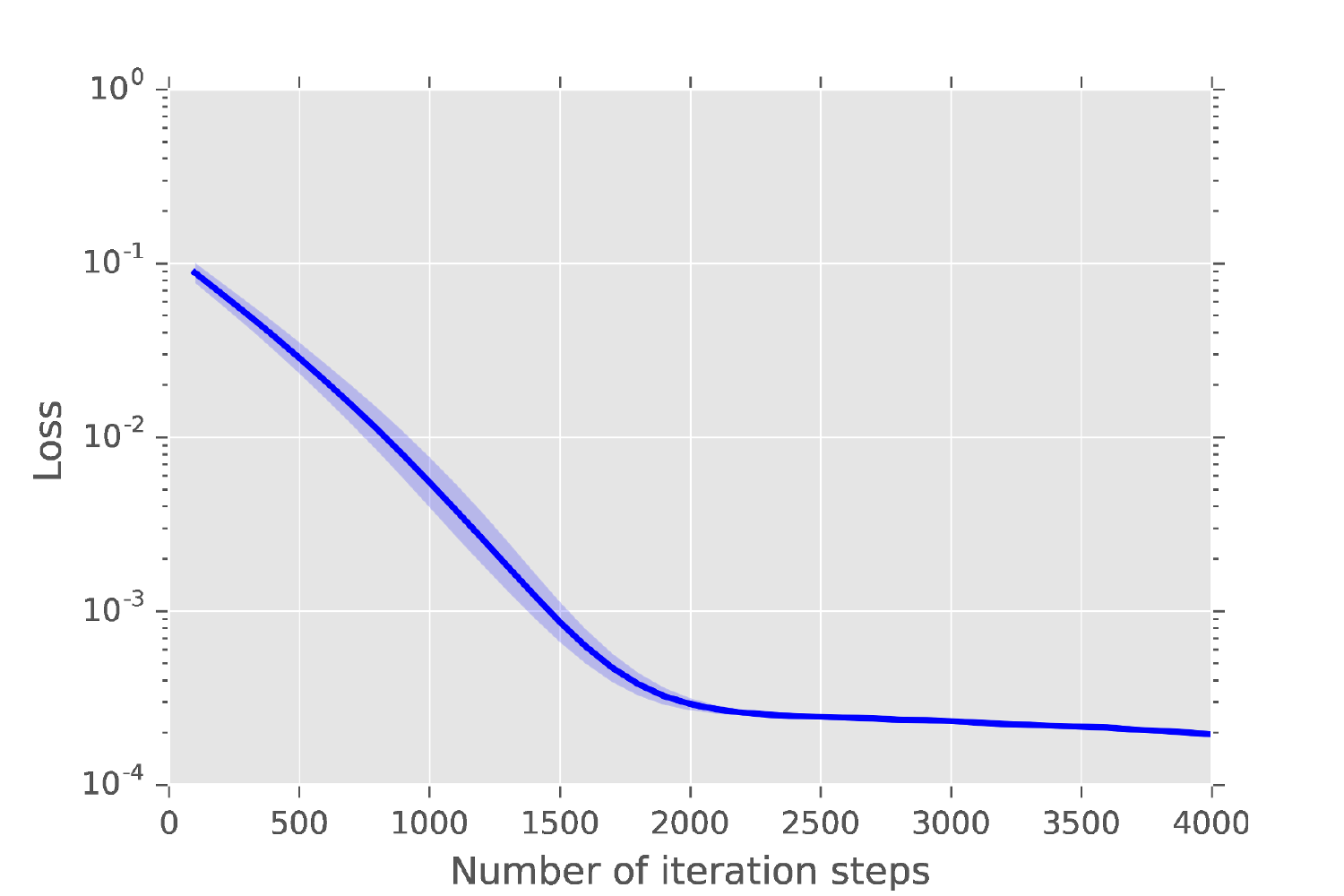}}
\caption{Relative $ L^1 $-approximation error 
of $ \mathcal{U}^{ \Theta_m } $
and mean of the loss function 
against $ m \in \{ 1, 2, 3, \dots, 4000 \} $
in the case of the PDE~\eqref{eq:PDE_allencahn}.
The deep BSDE approximation $ \mathcal{U}^{ \Theta_{ 4000 } } \approx u(0,\xi) $ achieves 
a relative $ L^1 $-approximation error of size $ 0.0030 $ in a runtime of $ 595 $ seconds.\label{fig:allen_cahn}}
\end{figure}

\subsection{A Hamilton-Jacobi-Bellman (HJB) equation}
\label{subsec:HJB}

In this subsection we apply the deep BSDE solver in Subsection~\ref{sec:general_case} 
to a Hamilton-Jacobi-Bellman (HJB) equation which admits an explicit 
solution that can be obtained through the Cole-Hopf transformation 
(cf., e.g., Chassagneux \& Richou~\cite[Section~4.2]{MR3449318}
and Debnath~\cite[Section~8.4]{MR2858125}).

Assume the setting in the Subsection~\ref{sec:example_setting} and
assume for all
$
  s, t \in [0,T]
$, 
$
  x,
  w \in \R^d
$,
$
  y \in \R
$, 
$
  z \in \R^{ 1 \times d }
$,
$ m \in \N $
that
$ d = 100 $,
$ T = 1 $,
$ N = 20 $,
$
  \gamma_m = \frac{ 1 }{ 100 }
$,
$
  \mu(t,x) = 0
$,
$
  \sigma(t,x) w = \sqrt{2} \, w
$,
$
  \xi =  ( 0, 0, \dots, 0 ) \in \R^d
$,
$
  \Upsilon( s, t, x, w )
  =
  x + \sqrt{2} \, w
$,
$
  f(t,x,y,z) = - \| z \|^2_{ \R^{ 1 \times d } }
$,
and
$
  g(x) 
  = 
  \ln(
    \frac{ 1 }{ 2 } 
    \, [ 1 + \| x \|^2_{ \R^d } ]
  )
$.
Note that the solution $ u $ 
of the PDE \eqref{eq:PDE_numerics} then satisfies for all 
$
  t \in [0,T)
$, 
$
  x \in \R^d
$
that
$
  u(T,x) = g(x) 
$
and
\begin{equation}  
\label{eq:PDE_HJB}
  \frac{ \partial u}{ \partial t } (t,x)
  + 
  ( \Delta_x u )(t,x)
  =
  \|
    ( \nabla_x u )(t,x)
  \|^2_{ \R^d }
  .
\end{equation}
In Table~\ref{tab:2} we approximatively 
calculate the mean 
of 
$ \mathcal{U}^{ \Theta_m } $, 
the standard deviation
of 
$ \mathcal{U}^{ \Theta_m } $, 
the relative $ L^1 $-approximatin error
associated to 
$ \mathcal{U}^{ \Theta_m } $,
the standard deviation of 
the relative $ L^1 $-approximatin error
associated to 
$ \mathcal{U}^{ \Theta_m } $, 
and the runtime 
in seconds needed to calculate 
one realization of $ \mathcal{U}^{ \Theta_m } $
against 
$ m \in \{ 0, 500, 1000, 1500, 2000 \} $
based on $ 5 $ independent realizations
($ 5 $ independent runs).
Table~\ref{tab:2} also depicts 
the mean of the loss function associated to 
$ \Theta_m $ 
and the standard deviation of the loss 
function associated to 
$ \Theta_m $ 
against 
$ m \in \{ 0, 500, 1000, 1500, 2000 \} $
based on $ 256 $ Monte Carlo samples 
and $ 5 $ independent realizations 
($ 5 $ independent runs).
In addition, 
the relative $ L^1 $-approximation error 
associated to
$ \mathcal{U}^{ \Theta_m } $
against 
$ m \in \{ 1, 2, 3, \dots, 2000 \} $
is pictured on the left hand side 
of Figure~\ref{fig:HJB}
based on $ 5 $ independent realizations ($ 5 $ independent runs)
and
the mean of the loss function 
associated to 
$ \Theta_m $ 
against 
$ m \in \{ 1, 2, 3, \dots, 2000 \} $
is pictured on the right hand side of 
Figure~\ref{fig:HJB}
based on $ 256 $ Monte Carlo samples 
and $ 5 $ independent realizations ($ 5 $ independent runs).
In the approximative computations of the relative $ L^1 $-approximation errors 
in Table~\ref{tab:2} and Figure~\ref{fig:HJB}
the value 
$
  u(0,\xi) = u(0,0,\dots, 0) 
$
of the exact solution $ u $
of the PDE~\eqref{eq:PDE_allencahn} is replaced by the value 
$
  4.5901 
$
which, in turn, is calculated by means 
of Lemma~\ref{lem:linearMC} below 
(with 
$ d = 100 $,
$ T = 1 $,
$ \alpha = 1 $,
$ \beta = -1 $,
$ 
  g = \R^d \ni x \mapsto 
  \ln(
    \frac{ 1 }{ 2 } 
    \, [ 1 + \| x \|^2_{ \R^d } ]
  )
  \in 
  \R
$
in the notation of Lemma~\ref{lem:linearMC})
and a classical Monte Carlo method
(see the {\sc Matlab} code~\ref{code:linearMC} below).
\begin{center}
\begin{table}[!ht]
\begin{center}
\begin{tabular}{|c|c|c|c|c|c|c|c|}
\hline
Number 
&
Mean
&
Standard
&
Relative
&
Standard
&
Mean
&
Standard
&
Runtime
\\
of
&
of
$ \mathcal{U}^{ \Theta_m } $
&
deviation
&
$ L^1 $-appr.
&
deviation
&
of the 
&
deviation
&
in sec.
\\
iteration
&
&
of $ \mathcal{U}^{ \Theta_m } $
&
error
&
of the
&
loss
&
of the 
&
for one
\\
steps $ m $
&
&
&
&
relative
&
function
&
loss
&
realization
\\
&
&
&
&
$ L^1 $-appr.
&
&
function
&
of $ \mathcal{U}^{ \Theta_m } $
\\
&
&
&
&
error
&
&
&
\\
\hline
0
&
0.3167
&
0.3059
&
0.9310
&
0.0666
&
18.4052
&
2.5090
&

\\
\hline
500
&
2.2785
&
0.3521
&
0.5036
&
0.0767
&
2.1789
&
0.3848
&
116
\\
\hline
1000
&
3.9229
&
0.3183
&
0.1454
&
0.0693
&
0.5226
&
0.2859
&
182
\\
\hline
1500
&
4.5921
&
0.0063
&
0.0013
&
0.006
&
0.0239
&
0.0024
&
248
\\
\hline
2000
&
4.5977
&
0.0019
&
0.0017
&
0.0004
&
0.0231
&
0.0026
&
330
\\
\hline
\end{tabular}
\end{center}
\caption{Numerical simulations for the deep BSDE solver 
in Subsection~\ref{sec:general_case} in the case of the PDE~\eqref{eq:PDE_HJB}. 
\label{tab:2}}
\end{table}
\end{center}
\begin{figure}[ht]
\centering
\setcounter{subfigure}{0}
\subfigure[Relative $ L^1 $-approximation error]{\includegraphics[width=8cm]{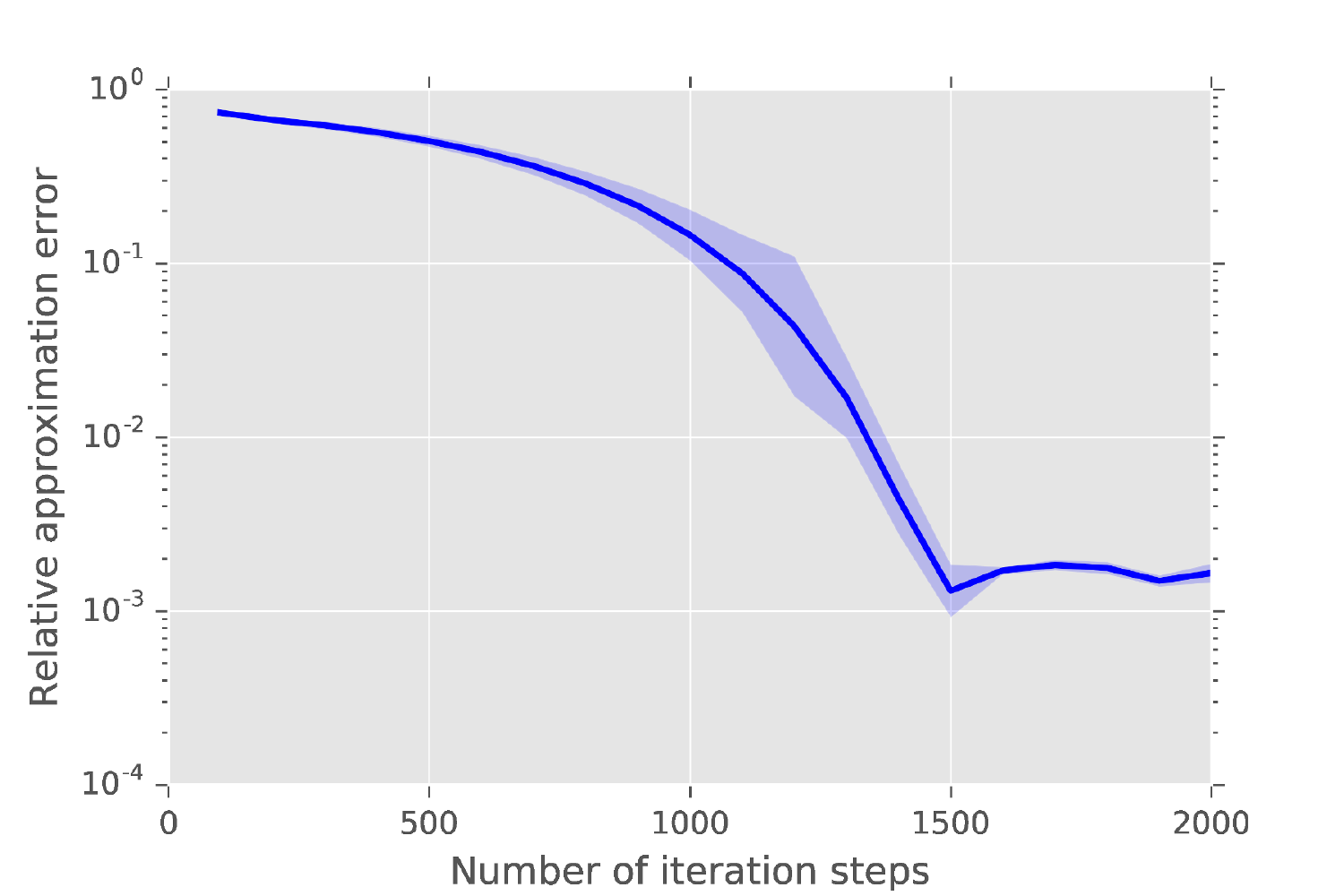}}
\subfigure[Mean of the loss function]{\includegraphics[width=8cm]{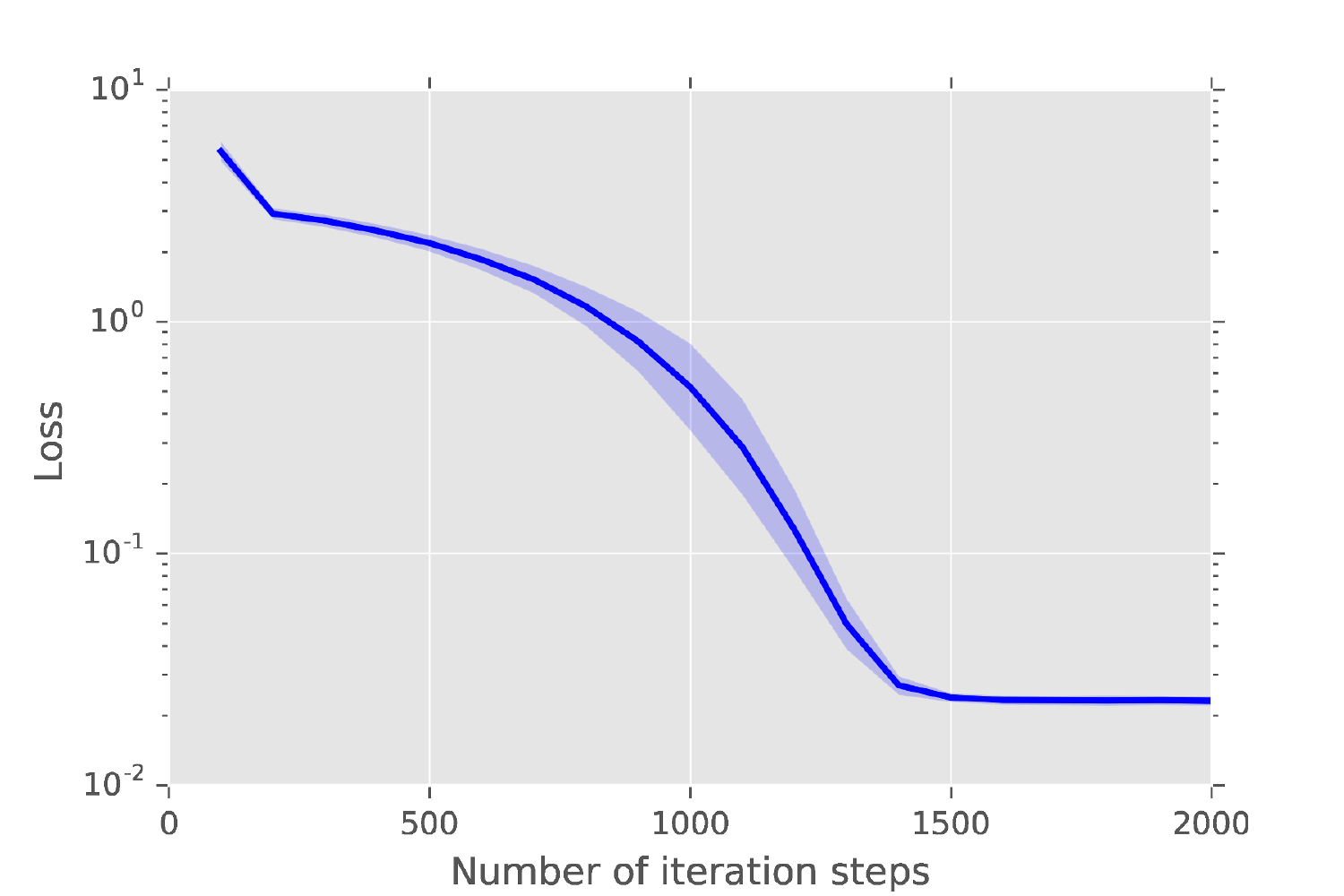}}
\caption{Relative $ L^1 $-approximation error 
of $ \mathcal{U}^{ \Theta_m } $
and mean of the loss function against $ m \in \{ 1, 2, 3, \dots, 2000 \} $.
The deep BSDE approximation $ \mathcal{U}^{ \Theta_{ 2000 } } \approx u(0,\xi) $ 
achieves a relative $ L^1 $-approximation error 
of size $ 0.0017 $ in a runtime of $ 283 $ seconds.
\label{fig:HJB}}
\end{figure}
\begin{lemma}[Cf., e.g., Section~4.2 in \cite{MR3449318} and Section~8.4 in \cite{MR2858125}]
\label{lem:linearMC}
Let $ d \in \N $, $ T, \alpha \in (0,\infty) $,
$ \beta \in \R \backslash \{ 0 \} $,
let $ ( \Omega, \mathcal{F}, \P ) $ be a probability space,
let $ W \colon [0,T] \times \Omega \to \R^d $ be a $ d $-dimensional standard Brownian motion,
let 
$ g 
\in C^2( \R^d, \R ) $
be a function which satisfies
$
  \sup_{ x \in \R^d } \left[ \beta g(x) \right] < \infty
$,
let $ f \colon [0,T] \times \R^d \times \R \times \R^d \to \R $
be the function which satisfies for all 
$ t \in [0,T] $, $ x = (x_1, \dots, x_d), z = ( z_1, \dots, z_d) \in \R^d $, $ y \in \R $
that
\begin{equation}
  f( t, x, y, z ) 
  =
  \beta \| z \|^2_{ \R^d }
  =
  \beta \smallsum_{ i = 1 }^d | z_i |^2
  ,
\end{equation}
and let $ u \colon [0,T] \times \R^d \to \R $
be the function which satisfies 
for all 
$ (t,x) \in [0,T] \times \R^d $ 
that
\begin{equation}
  u(t,x) = 
  \frac{ \alpha }{ \beta } 
  \ln\!\bigg( 
    \E\Big[ 
      \exp\!\Big( 
        \frac{ \beta g( x + W_{ T - t } \sqrt{ 2 \alpha } ) }{ \alpha } 
      \Big) 
    \Big] 
  \bigg)  
  .
\end{equation}
Then 
\begin{enumerate}[(i)]
\item 
\label{item:lem_i}
it holds that
$ u \colon [0,T] \times \R^d \to \R $ is a continuous function,
\item 
\label{item:lem_ii}
it holds that
$
  u|_{
    [0,T) \times \R^d
  } 
  \in C^{ 1, 2 }( [0,T) \times \R^d, \R )
$,
and 
\item
\label{item:lem_iii}
it holds for all $ (t,x) \in [0,T) \times \R^d $ that
$
  u(T,x) 
  = 
  g(x)
$
and
\begin{equation}
\begin{split}
&
  \frac{ \partial u}{ \partial t } (t,x)
  + 
  \alpha
  ( \Delta_x u )(t,x)
  +
  f\big( 
    t, x, u(t,x), ( \nabla_x u )( t, x ) 
  \big)
\\
&
=
  \frac{ \partial u}{ \partial t } (t,x)
  + 
  \alpha ( \Delta_x u )(t,x)
  +
  \beta 
  \|
     ( \nabla_x u )(t,x)
  \|^2_{ \R^d }
\\
&
=
  \frac{ \partial u}{ \partial t } (t,x)
  + 
  \alpha ( \Delta_x u )(t,x)
  +
  \beta 
  \sum_{ j = 1 }^d
  \Big|
      \frac{ \partial u}{ \partial x_j } 
    (t,x)
  \Big|^2
  = 0
  .
\end{split}
\end{equation}
\end{enumerate}
\end{lemma}

\begin{proof}[Proof of Lemma~\ref{lem:linearMC}]
Throughout this proof let 
$ c = \frac{ \alpha }{ \beta } \in \R \backslash \{ 0 \} $
and let 
$
  \mathcal{V} \colon \R^d \to (0,\infty)
$
and 
$ v \colon [0,T] \times \R^d \to (0,\infty) $
be the functions which satisfy 
for all $ t \in [0,T] $, $ x \in \R^d $ 
that
\begin{equation}
\begin{split} 
  \mathcal{V}( x ) 
 =
      \exp\!\Big( 
        \frac{ g( x ) }{ c } 
      \Big) 
  =
      \exp\!\Big( 
        \frac{ \beta g( x ) }{ \alpha } 
      \Big) 
  \qquad 
  \text{and}
  \qquad
    v(t,x) 
=
    \E\big[ 
      \mathcal{V}\big( 
        x + W_{ T - t } \sqrt{ 2 \alpha } 
      \big) 
    \big] 
  .
\end{split}
\end{equation}
Observe that the hypothesis 
that 
$
      \sup_{ x \in \R^d }
      \left[ 
        \beta g(x)
      \right]
      < \infty
$
ensures that for all $ \omega \in \Omega $ it holds that
\begin{equation}
\begin{split}
    \sup_{ t \in [0,T] }
    \sup_{ x \in \R^d }
    \big|
      \mathcal{V}\big( 
        x + W_{ T - t }(\omega) \sqrt{ 2 \alpha } 
      \big) 
    \big|
& \leq
    \sup_{ x \in \R^d }
    |
      \mathcal{V}( 
        x 
      ) 
    |
  =
    \sup_{ x \in \R^d }
      \mathcal{V}( 
        x 
      ) 
\\ &
  =
  \exp\!\left(
    \frac{
      \sup_{ x \in \R^d }
      \left[ 
        \beta g(x)
      \right]
    }{
      \alpha
    }
  \right)
  < \infty
  .
\end{split}
\end{equation}
Combining this with Lebesgue's theorem of dominated convergence 
ensures that $ v \colon [0,T] \times \R^d \to (0,\infty) $ is a continuous function. 
This and the fact that 
\begin{equation}
\label{eq:relation_v_u}
  \forall \, (t,x) \in [0,T] \times \R^d \colon
  u(t,x) = c \ln( v(t,x) )
\end{equation}
establish Item~\eqref{item:lem_i}. 
Next note that the Feynman-Kac formula ensures that
for all $ t \in [0,T) $, $ x \in \R^d $ 
it holds that
$
  v|_{ [0,T) \times \R^d }
  \in C^{ 1, 2 }( [0,T) \times \R^d , (0,\infty) )
$
and 
\begin{equation}
\label{eq:v_heat_equation}
  \frac{ \partial v }{ \partial t }( t, x )
  +
  \alpha
  ( \Delta_x v )( t, x )
  = 
  0
  .
\end{equation}
This and 
\eqref{eq:relation_v_u}
demonstrate Item~\eqref{item:lem_ii}.
It thus remains to prove Item~\eqref{item:lem_iii}.
For this note that the chain rule 
and \eqref{eq:relation_v_u} imply that
for all $ t \in [0,T) $, $ x = ( x_1, \dots, x_d ) \in \R^d $, 
$ i \in \{ 1, 2, \dots, d \} $
it holds that
\begin{equation}
\label{eq:u_time_derivative2}
  \frac{ \partial u }{ \partial t }(t,x)
  =
  \frac{ c }{ v(t,x) }
  \cdot 
  \frac{ \partial v }{ \partial t }(t,x)
\qquad
  \text{and}
\qquad
  \frac{ \partial u }{ \partial x_i }(t,x)
  =
  \frac{ c }{ v(t,x) }
  \cdot 
  \frac{ \partial v }{ \partial x_i }(t,x)
  .
\end{equation}
Again the chain rule and \eqref{eq:relation_v_u} hence 
ensure that
for all $ t \in [0,T) $, $ x = ( x_1, \dots, x_d ) \in \R^d $, 
$ i \in \{ 1, 2, \dots, d \} $
it holds that
\begin{equation}
  \frac{ \partial u^2 }{ \partial x_i^2 }(t,x)
  =
  \frac{ c }{ v(t,x) }
  \cdot 
  \frac{ \partial v^2 }{ \partial x_i^2 }(t,x)
  -
  \frac{ c }{ \left[ v(t,x) \right]^2 }
  \cdot 
  \left[
    \frac{ \partial v }{ \partial x_i }(t,x)
  \right]^2
  .
\end{equation}
This assures that
for all $ t \in [0,T) $, $ x = ( x_1, \dots, x_d ) \in \R^d $, 
$ i \in \{ 1, 2, \dots, d \} $
it holds that
\begin{equation}
\begin{split}
  \alpha ( \Delta_x u )(t,x)
& =
  \frac{ \alpha c }{ v(t,x) }
  \cdot 
  ( \Delta_x v )( t, x )
  -
  \frac{ \alpha c }{ \left[ v(t,x) \right]^2 }
  \cdot 
  \sum_{ i = 1 }^d
  \left[
    \frac{ \partial v }{ \partial x_i }(t,x)
  \right]^2
\\ 
& 
=
  \frac{ 
    \alpha c 
    ( \Delta_x v )( t, x )
  }{ v(t,x) }
  -
  \frac{ 
    \alpha c 
    \left\|
      ( \nabla_x v )( t, x )
    \right\|^2_{ \R^d }
  }{ \left[ v(t,x) \right]^2 }
  .
\end{split}
\end{equation}
Combining this with \eqref{eq:u_time_derivative2} 
demonstrates that 
for all $ t \in [0,T) $, $ x \in \R^d $
it holds that
\begin{equation}
\begin{split}
&
  \frac{ \partial u }{ \partial t }(t,x)
  +
  \alpha ( \Delta_x u )(t,x)
  +
  \beta
  \left\|
    ( \nabla_x u )( t, x )
  \right\|^2_{ \R^d }
\\
&
=
  \frac{ c }{ v(t,x) }
  \cdot 
  \frac{ \partial v }{ \partial t }(t,x)
  +
  \frac{ 
    \alpha c 
    ( \Delta_x v )( t, x )
  }{ v(t,x) }
  -
  \frac{ 
    \alpha c 
    \left\|
      ( \nabla_x v )( t, x )
    \right\|^2_{ \R^d }
  }{ \left[ v(t,x) \right]^2 }
  +
  \beta
  \left\|
    ( \nabla_x u )( t, x )
  \right\|^2_{ \R^d }
  .
\end{split}
\end{equation}
Equation~\eqref{eq:v_heat_equation} hence shows that
for all $ t \in [0,T) $, $ x = ( x_1, \dots, x_d ) \in \R^d $
it holds that
\begin{equation}
\begin{split}
&
  \frac{ \partial u }{ \partial t }(t,x)
  +
  \alpha ( \Delta_x u )(t,x)
  +
  \beta
  \left\|
    ( \nabla_x u )( t, x )
  \right\|^2_{ \R^d }
\\
&
=
  \beta
  \left\|
    ( \nabla_x u )( t, x )
  \right\|^2_{ \R^d }
  -
  \frac{ 
    \alpha c 
    \left\|
      ( \nabla_x v )( t, x )
    \right\|^2_{ \R^d }
  }{ \left[ v(t,x) \right]^2 }
\\ &
=
  \beta
  \left[
    \sum_{ i = 1 }^d
    \left|
      \frac{ \partial u }{ \partial x_i }( t, x )
    \right|^2
  \right]
  -
  \frac{ 
    \alpha c 
    \left\|
      ( \nabla_x v )( t, x )
    \right\|^2_{ \R^d }
  }{ \left[ v(t,x) \right]^2 }
  .
\end{split}
\end{equation}
This and \eqref{eq:u_time_derivative2} demonstrate that
for all $ t \in [0,T) $, $ x = ( x_1, \dots, x_d ) \in \R^d $
it holds that
\begin{equation}
\begin{split}
&
  \frac{ \partial u }{ \partial t }(t,x)
  +
  \alpha ( \Delta_x u )(t,x)
  +
  \beta
  \left\|
    ( \nabla_x u )( t, x )
  \right\|^2_{ \R^d }
\\ &
=
  \beta
  \left[
    \sum_{ i = 1 }^d
    \left|
      \frac{ c }{ v(t,x) } 
      \cdot 
      \frac{ \partial v }{ \partial x_i }( t, x )
    \right|^2
  \right]
  -
  \frac{ 
    \alpha c 
    \left\|
      ( \nabla_x v )( t, x )
    \right\|^2_{ \R^d }
  }{ \left[ v(t,x) \right]^2 }
\\ &
=
  \frac{ 
    c^2 \beta 
  }{ 
    \left[ v(t,x) \right]^2 
  }
  \left[
    \sum_{ i = 1 }^d
    \left|
      \frac{ \partial v }{ \partial x_i }( t, x )
    \right|^2
  \right]
  -
  \frac{ 
    \alpha c 
    \left\|
      ( \nabla_x v )( t, x )
    \right\|^2_{ \R^d }
  }{ 
    \left[ v(t,x) \right]^2 
  }
\\ &
  =
  \frac{
    \left[ 
      c^2 \beta - c \alpha
    \right] 
    \left\|
      ( \nabla_x v )( t, x )
    \right\|^2_{ \R^d }
  }{ 
    \left[ v(t,x) \right]^2 
  }
  = 0
  .
\end{split}
\end{equation}
This and the fact that 
\begin{equation}
  \forall \, x \in \R^d \colon
  u(T,x)
  =
  c \ln( v(T,x) )
  =
  c \ln( \mathcal{V}(x) )
  =
  c 
  \ln\!\left(
    \exp\!\left(
      \frac{ g(x) }{ c }
    \right)
  \right)
  =
  g(x)
\end{equation}
establish Item~\eqref{item:lem_iii}. 
The proof of Lemma~\ref{lem:linearMC} is thus completed.
\end{proof}

\subsection{Pricing of European financial derivatives with different interest rates 
for borrowing and lending}
\label{subsec:borrowlend}

In this subsection we apply the deep BSDE solver to a pricing problem 
of an European financial derivative in a financial market where the risk 
free bank account used for the hedging of the financial 
derivative has different interest rates for borrowing and lending
(see Bergman~\cite{Bergman1995} and, 
e.g., \cite{GobetLemorWarin2005,BenderDenk2007,
BenderSchweizerZhuo2014,BriandLabart2014,CrisanManolarakis2012,Eetal2017}
where this example has been used as a test example 
for numerical methods for BSDEs).

Assume the setting in Subsection~\ref{sec:example_setting}, 
let 
$ 
  \bar{\mu} = \frac{ 6 }{ 100 } 
$,
$
  \bar{\sigma} = \frac{ 2 }{ 10 } 
$, 
$
  R^l = \frac{ 4 }{ 100 } 
$, 
$
  R^b = \frac{ 6 }{ 100 } 
$, 
and assume 
for all
$
  s, t \in [0,T]
$, 
$
  x = ( x_1, \dots, x_d ),
  w = ( w_1, \dots, w_d ) 
  \in \R^d
$, 
$
  y \in \R
$, 
$
  z \in \R^{ 1 \times d } 
$,
$ m \in \N $
that 
$
  d = 100
$,
$
  T = \nicefrac{ 1 }{ 2 } 
$,
$ N = 20 $,
$ 
  \gamma_m = 5 \cdot 10^{ - 3 } = 0.005 
$,
$
  \mu( t, x ) = \bar{\mu} x
$,
$
  \sigma( t, x ) = 
  \bar{\sigma} \operatorname{diag}_{ \R^{ d \times d } }( x_1, \dots, x_d ) 
$,
$
  \xi = ( 100, 100, \dots, 100 ) \in \R^d
$,
and
\begin{equation}
  g(x) = 
  \max\left\{
    \left[ 
      \max_{ 1 \leq i \leq 100 } x_i 
    \right] 
    - 120 , 0 
  \right\} 
  - 
  2 
  \max\!\left\{ 
    \left[
      \max_{ 1 \leq i \leq 100 } x_i 
    \right]
    - 150 , 0
  \right\} 
  ,
\end{equation}
\begin{equation}
  \Upsilon( s, t, x, w )
  =
  \exp\!\left(
    \left( \bar{\mu} - \frac{ \bar{\sigma}^2 }{ 2 } \right) ( t - s )
  \right)
  \exp\!\big(
    \bar{\sigma} 
    \operatorname{diag}_{ \R^{ d \times d } }( w_1, \dots, w_d )
  \big)
  \,
  x
  ,
\end{equation}
\begin{equation}
  f(t,x,y,z)
  =
  - 
  R^l y 
  - 
  \frac{ ( \bar{\mu} - R^l ) }{ \bar{\sigma} }
  \sum_{ i = 1 }^d z_i
  +
  ( R^b - R^l )
  \max\!\left\{
    0 ,
    \left[ 
      \frac{ 1 }{ \bar{\sigma} } 
      \sum_{ i = 1 }^d z_i
    \right]
    - y
  \right\}.
\end{equation}
Note that the solution $ u $ of the 
PDE~\eqref{eq:PDE_numerics}
then satisfies for all 
$ 
  t \in [0,T) 
$, 
$ 
  x = (x_1, x_2, \ldots, x_d) \in \R^d 
$
that
$
  u(T,x) = g(x) 
$
and
\begin{multline}  
\label{eq:PDE_borrowlend_not_explicit}
   \frac{ \partial u}{ \partial t } ( t, x )
  + 
  f\big( 
    t, x, 
    u(t,x),
    \bar{\sigma} 
    \operatorname{diag}_{ \R^{ d \times d } }(x_1, \dots, x_d) 
    ( \nabla_x u )( t, x )
  \big)
  +
  \bar{\mu} \sum_{i=1}^d x_i \, \frac{\partial u}{\partial x_i}(t,x)
\\
  +
  \frac{\bar{\sigma}^2}{ 2 }
  \sum_{i=1}^d
  | x_i |^2
  \, 
    \frac{ \partial^2 u}{ \partial x^2_i }(t,x)
  = 0
  .
\end{multline}
Hence, we obtain
for all $ t \in [0,T) $, 
$ x = (x_1,x_2,\ldots,x_d)\in \R^d $
that
$
  u(T,x) = g(x) 
$
and
\begin{multline}  
\label{eq:PDE_borrowlend_max}
  \frac{ \partial u}{ \partial t } ( t, x )
  +
  \frac{ \bar{\sigma}^2 }{ 2 }
  \sum_{ i = 1 }^d
  |x_i|^2
  \,
    \frac{ \partial^2 u}{ \partial x^2_i }( t, x ) 
\\
  +
  \max\!\left\{
    R^b 
    \left(
    \left[
      \smallsum_{ i = 1 }^d x_i 
      \,
      \big(
        \frac{ \partial u}{ \partial x_i } 
      \big)(t,x)
    \right]
      -
      u(t,x) 
    \right)
    ,
    R^l
    \left( 
      \left[
      \smallsum_{ i = 1 }^d 
      x_i 
      \,
      \big(
        \frac{ \partial u}{ \partial x_i } 
      \big)(t,x)
      \right]
      -
      u(t,x)
    \right)
  \right\}
  = 0
  .  
\end{multline}
This shows that
for all $ t \in [0,T) $, 
$ x = (x_1,x_2,\ldots,x_d)\in \R^d $
it holds that
$
  u(T,x) = g(x) 
$
and
\begin{multline}  
\label{eq:PDE_borrowlend}
  \frac{ \partial u}{ \partial t } ( t, x )
  +
  \frac{ \bar{\sigma}^2 }{ 2 }
  \sum_{ i = 1 }^d
  |x_i|^2
  \,
    \frac{ \partial^2 u}{ \partial x^2_i } 
  (t,x)
\\
  -
  \min\!\bigg\{
    R^b \bigg(
      u(t,x) - 
      \sum_{ i = 1 }^d x_i 
      \,
        \frac{ \partial u}{ \partial x_i } 
      (t,x)
    \bigg)
    ,
    R^l
    \bigg( 
      u(t,x) - 
      \sum_{ i = 1 }^d 
      x_i 
      \,
        \frac{ \partial u}{ \partial x_i } 
      (t,x)
    \bigg)
  \bigg\}
  = 0
  .  
\end{multline}
In Table~\ref{tab:3} we approximatively 
calculate the mean 
of 
$ \mathcal{U}^{ \Theta_m } $, 
the standard deviation
of 
$ \mathcal{U}^{ \Theta_m } $, 
the relative $ L^1 $-approximatin error
associated to 
$ \mathcal{U}^{ \Theta_m } $,
the standard deviation of 
the relative $ L^1 $-approximatin error
associated to 
$ \mathcal{U}^{ \Theta_m } $, 
and the runtime 
in seconds needed to calculate 
one realization of $ \mathcal{U}^{ \Theta_m } $
against 
$ m \in \{ 0, 1000, 2000, 3000, 4000 \} $
based on $ 5 $ independent realizations
($ 5 $ independent runs).
Table~\ref{tab:3} also depicts 
the mean of the loss function associated to 
$ \Theta_m $ 
and the standard deviation of the loss 
function associated to 
$ \Theta_m $ 
against 
$ m \in \{ 0, 1000, 2000, 3000, 4000 \} $
based on $ 256 $ Monte Carlo samples 
and $ 5 $ independent realizations 
($ 5 $ independent runs).
In addition, 
the relative $ L^1 $-approximation error 
associated to
$ \mathcal{U}^{ \Theta_m } $
against 
$ m \in \{ 1, 2, 3, \dots, 4000 \} $
is pictured on the left hand side 
of Figure~\ref{fig:borrowlend}
based on $ 5 $ independent realizations ($ 5 $ independent runs)
and
the mean of the loss function 
associated to 
$ \Theta_m $ 
against 
$ m \in \{ 1, 2, 3, \dots, 4000 \} $
is pictured on the right hand side of 
Figure~\ref{fig:borrowlend}
based on $ 256 $ Monte Carlo samples 
and $ 5 $ independent realizations ($ 5 $ independent runs).
In the approximative computations of the relative $ L^1 $-approximation errors 
in Table~\ref{tab:3} and Figure~\ref{fig:borrowlend}
the value 
$
  u(0,\xi) = u(0,0,\dots, 0) 
$
of the exact solution $ u $
of the PDE~\eqref{eq:PDE_borrowlend} 
is replaced by the value
$ 21.299 $
which, in turn, is calculated by means 
of the multilevel-Picard approximation method 
in E et al.~\cite{Eetal2017}
(see \cite[$ \rho = 7 $ in Table~6 in Section~4.3]{Eetal2017}).
\begin{center}
\begin{table}[!ht]
\begin{center}
\begin{tabular}{|c|c|c|c|c|c|c|c|}
\hline
Number 
&
Mean
&
Standard
&
Relative
&
Standard
&
Mean
&
Standard
&
Runtime
\\
of
&
of
$ \mathcal{U}^{ \Theta_m } $
&
deviation
&
$ L^1 $-appr.
&
deviation
&
of the 
&
deviation
&
in sec.
\\
iteration
&
&
of $ \mathcal{U}^{ \Theta_m } $
&
error
&
of the
&
loss
&
of the 
&
for one
\\
steps $ m $
&
&
&
&
relative
&
function
&
loss
&
realization
\\
&
&
&
&
$ L^1 $-appr.
&
&
function
&
of $ \mathcal{U}^{ \Theta_m } $
\\
&
&
&
&
error
&
&
&
\\
\hline
0
&
16.964
&
0.882
&
0.204
&
0.041
&
53.666
&
8.957
&

\\
\hline
1000
&
20.309
&
0.524
&
0.046
&
0.025
&
30.886
&
3.076
&
194
\\
\hline
2000
&
21.150
&
0.098
&
0.007
&
0.005
&
29.197
&
3.160
&
331
\\
\hline
3000
&
21.229
&
0.034
&
0.003
&
0.002
&
29.070
&
3.246
&
470
\\
\hline
4000
&
21.217
&
0.043
&
0.004
&
0.002
&
29.029
&
3.236
&
617
\\
\hline
\end{tabular}
\end{center}
\caption{Numerical simulations for the deep BSDE solver 
in Subsection~\ref{sec:general_case} in the case of the PDE~\eqref{eq:PDE_borrowlend}.
\label{tab:3}}
\end{table}
\end{center}
\begin{figure}[ht]
\centering
\setcounter{subfigure}{0}
\subfigure[Relative $ L^1 $-approximation error]{\includegraphics[width=8cm]{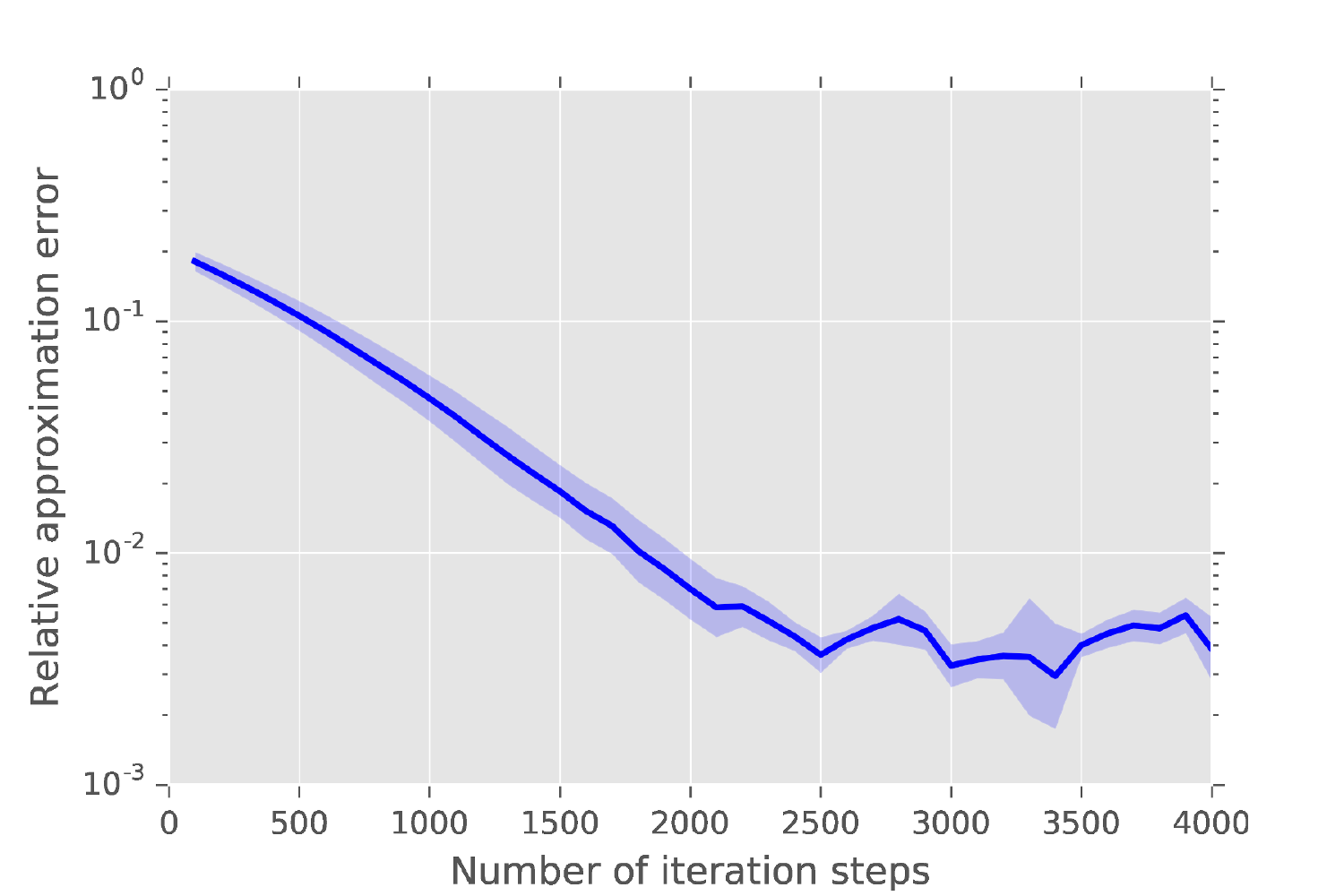}}
\subfigure[Mean of the loss function]{\includegraphics[width=8cm]{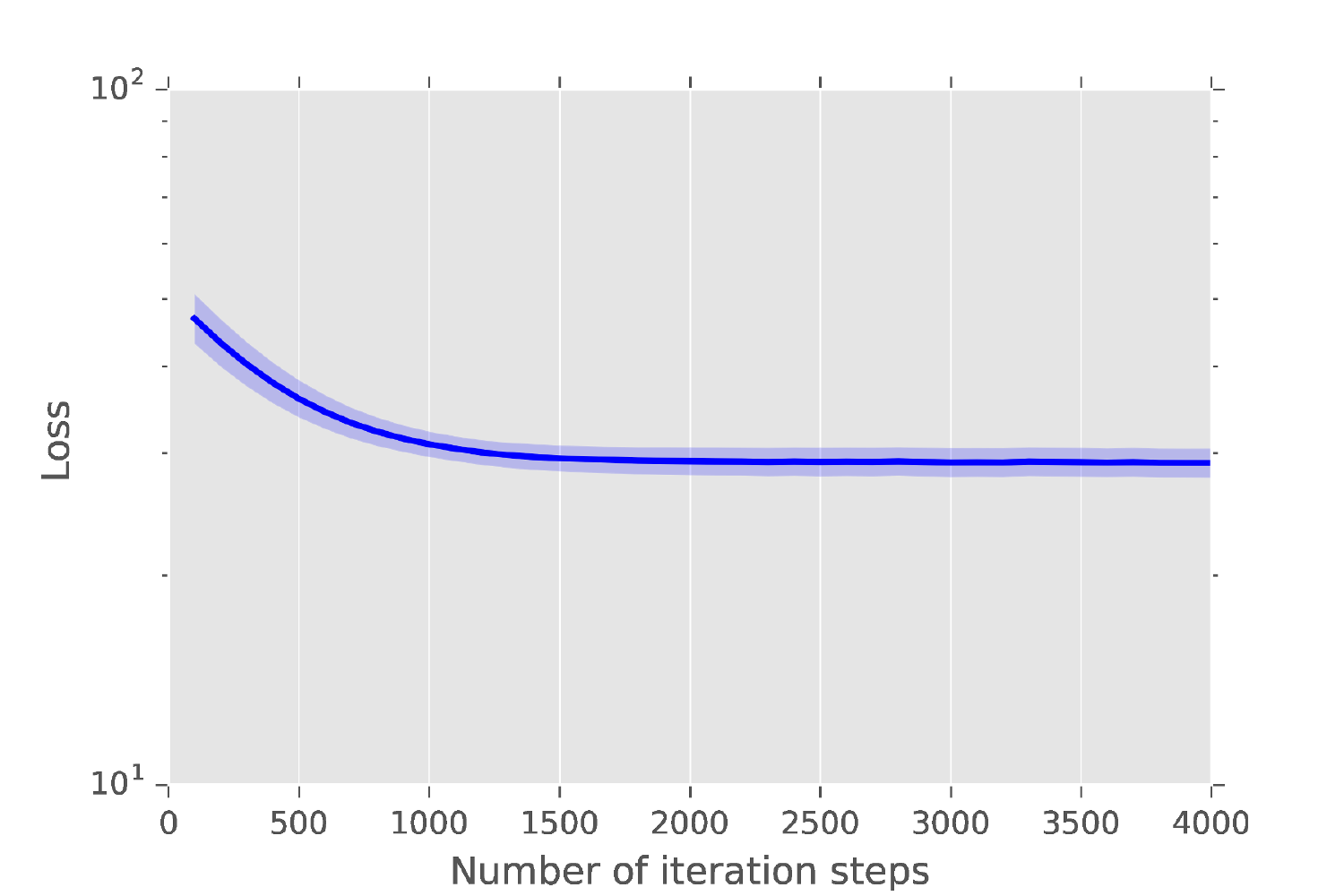}}
\caption{Relative $ L^1 $-approximation error 
of $ \mathcal{U}^{ \Theta_m } $
and mean of the loss function against $ m \in \{ 1, 2, 3, \dots, 4000 \} $
in the case of the PDE~\eqref{eq:PDE_borrowlend}.
The deep BSDE approximation $ \mathcal{U}^{ \Theta_{ 4000 } } \approx u(0,\xi) $ 
achieves a relative $ L^1 $-approximation error 
of size $ 0.0039 $ in a runtime of $ 566 $ seconds.
\label{fig:borrowlend}}
\end{figure}

\subsection{Multidimensional Burgers-type PDEs with explicit solutions}

In this subsection we consider a high-dimensional version 
of the example analyzed numerically in 
Chassagneux~\cite[Example~4.6 in Subsection~4.2]{MR3284573}.

More specifically, assume the setting in Subsection~\ref{sec:example_setting}, 
and assume 
for all
$
  s, t \in [0,T]
$, 
$
  x = ( x_1, \dots, x_d ),
  w = ( w_1, \dots, w_d ) 
  \in \R^d
$, 
$
  y \in \R
$, 
$
  z = ( z_i )_{ i \in \{ 1, 2, \dots, d \} } \in \R^{ 1 \times d } 
$
that 
$
  \mu( t, x ) = 0
$,
$
  \sigma( t, x ) w = 
  \frac{ d }{ \sqrt{2} } w
$,
$
  \xi = ( 0, 0, \dots, 0 ) \in \R^d
$,
$
  \Upsilon( s, t, x, w )
  =
  x + \frac{ d }{ \sqrt{2} } w
$,
and
\begin{equation}
  g(x) 
  =   
  \frac{ 
    \exp( 
      T + \frac{ 1 }{ d } \sum_{ i = 1 }^d x_i 
    )  
  }{ 
    \big( 
      1 + \exp( T + \frac{ 1 }{ d } \sum_{ i = 1 }^d x_i ) 
    \big) 
  }
  ,
\qquad
  f(t,x,y,z)
  =
  \left(
    y - \frac{ 2 + d }{ 2 d }
  \right)
  \left(
    \smallsum\limits_{ i = 1 }^d z_i
  \right)
  .
\end{equation}
Note that the solution $ u $ of the 
PDE~\eqref{eq:PDE_numerics}
then satisfies for all 
$ 
  t \in [0,T) 
$, 
$ 
  x = (x_1, x_2, \ldots, x_d) \in \R^d 
$
that
$
  u(T,x) = g(x) 
$
and
\begin{equation}
\label{eq:PDE_JF}
\frac{\partial u}{\partial t}(t,x)
  +
  \frac{d^2}{2} ( \Delta_x u )(t,x)
  +
  \left(
    u(t,x) - \frac{ 2 + d }{ 2d }
  \right)
  \left(
    d \smallsum\limits_{ i = 1 }^d 
    \displaystyle
    \frac{ \partial u}{ \partial x_i }( t, x ) 
  \right) = 0
\end{equation}
(cf.\ Lemma~\ref{lem:JF} below 
[with $ \alpha = d^2 $, $ \kappa = \nicefrac{ 1 }{ d } $
in the notation of Lemma~\ref{lem:JF} below]).
%
%
On the left hand side of Figure~\ref{fig:JF}
we present approximatively 
the relative $ L^1 $-approximatin error
associated to 
$ \mathcal{U}^{ \Theta_m } $
against 
$ m \in \{ 1, 2, 3, \dots, 60 \, 000 \} $
based on $ 5 $ independent realizations
($ 5 $ independent runs)
in the case 
\begin{equation}
\label{eq:parameter_JF}
  T = 1
  ,
  \qquad
  d = 20
  ,
  \qquad
  N = 80
  ,
  \qquad
  \forall \, m \in \N
  \colon 
  \gamma_m 
  =
  10^{
    (
      \mathbbm{1}_{ [1,30000] }( m )
      +
      \mathbbm{1}_{ [1,50000] }( m )
      - 4
    )
  }
  .
\end{equation}
On the right hand side of Figure~\ref{fig:JF}
we present approximatively 
the mean of the loss function 
associated to 
$ \Theta_m $ 
against 
$ m \in \{ 1, 2, 3, \dots, 60 \, 000 \} $
based on $ 256 $ Monte Carlo samples 
and $ 5 $ independent realizations ($ 5 $ independent runs)
in the case \eqref{eq:parameter_JF}.
On the left hand side of Figure~\ref{fig:JFb}
we present approximatively 
the relative $ L^1 $-approximatin error
associated to 
$ \mathcal{U}^{ \Theta_m } $
against 
$ m \in \{ 1, 2, 3, \dots, 30 \, 000 \} $
based on $ 5 $ independent realizations
($ 5 $ independent runs)
in the case 
\begin{equation}
\label{eq:parameter_JFb}
  T = \frac{ 2 }{ 10 }
  ,
  \qquad
  d = 50
  ,
  \qquad
  N = 30
  ,
  \qquad
  \forall \, m \in \N
  \colon 
  \gamma_m 
  =
  10^{
    (
      \mathbbm{1}_{ [1,15000] }( m )
      +
      \mathbbm{1}_{ [1,25000] }( m )
      - 4
    )
  }
  .
\end{equation}
On the right hand side of Figure~\ref{fig:JFb}
we present approximatively 
the mean of the loss function 
associated to 
$ \Theta_m $ 
against 
$ m \in \{ 1, 2, 3, \dots, 30 \, 000 \} $
based on $ 256 $ Monte Carlo samples 
and $ 5 $ independent realizations ($ 5 $ independent runs)
in the case \eqref{eq:parameter_JFb}.

\begin{figure}[ht]
\centering
\setcounter{subfigure}{0}
\subfigure[Relative $ L^1 $-approximation error]{\includegraphics[width=8cm]{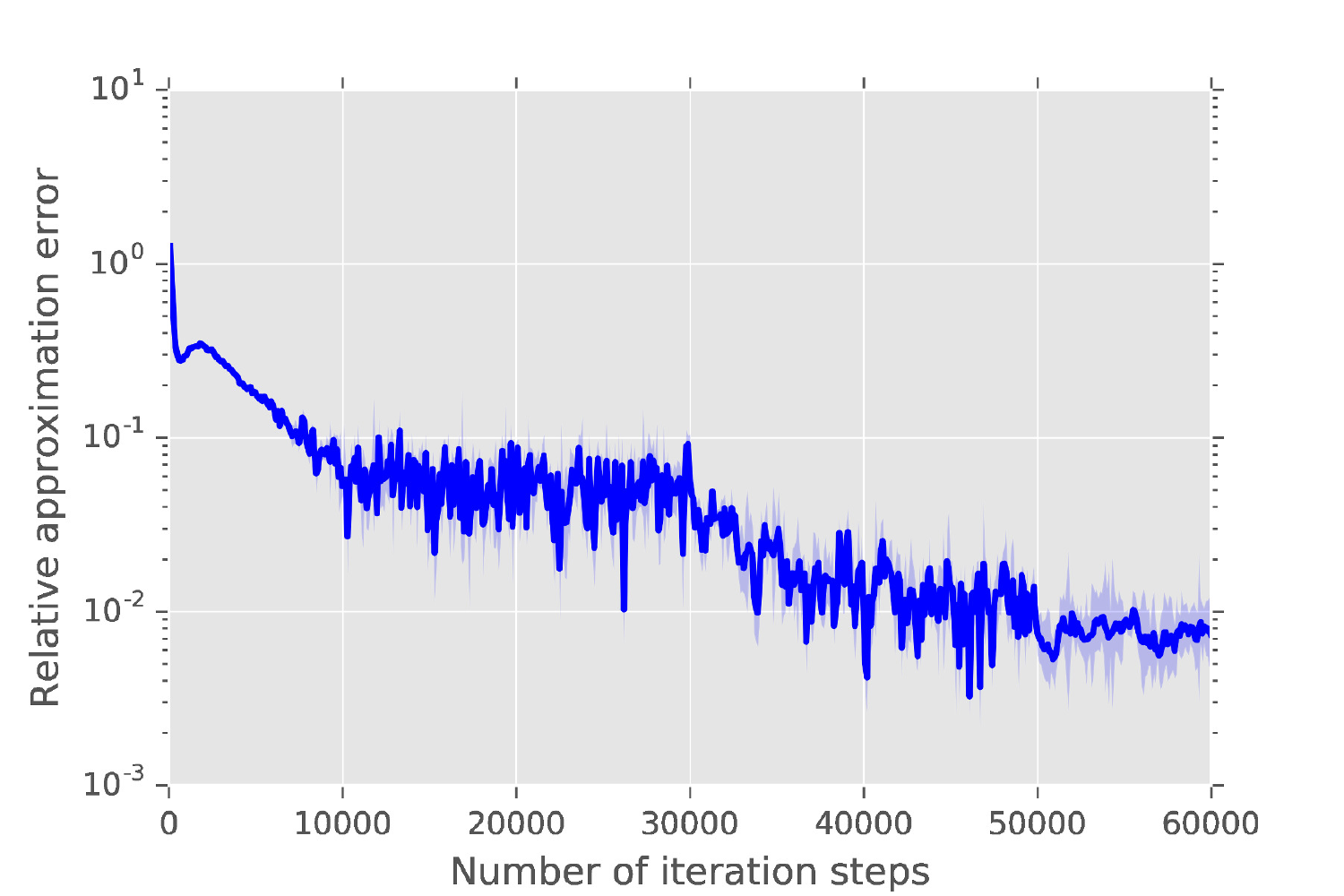}}
\subfigure[Mean of the loss function]{\includegraphics[width=8cm]{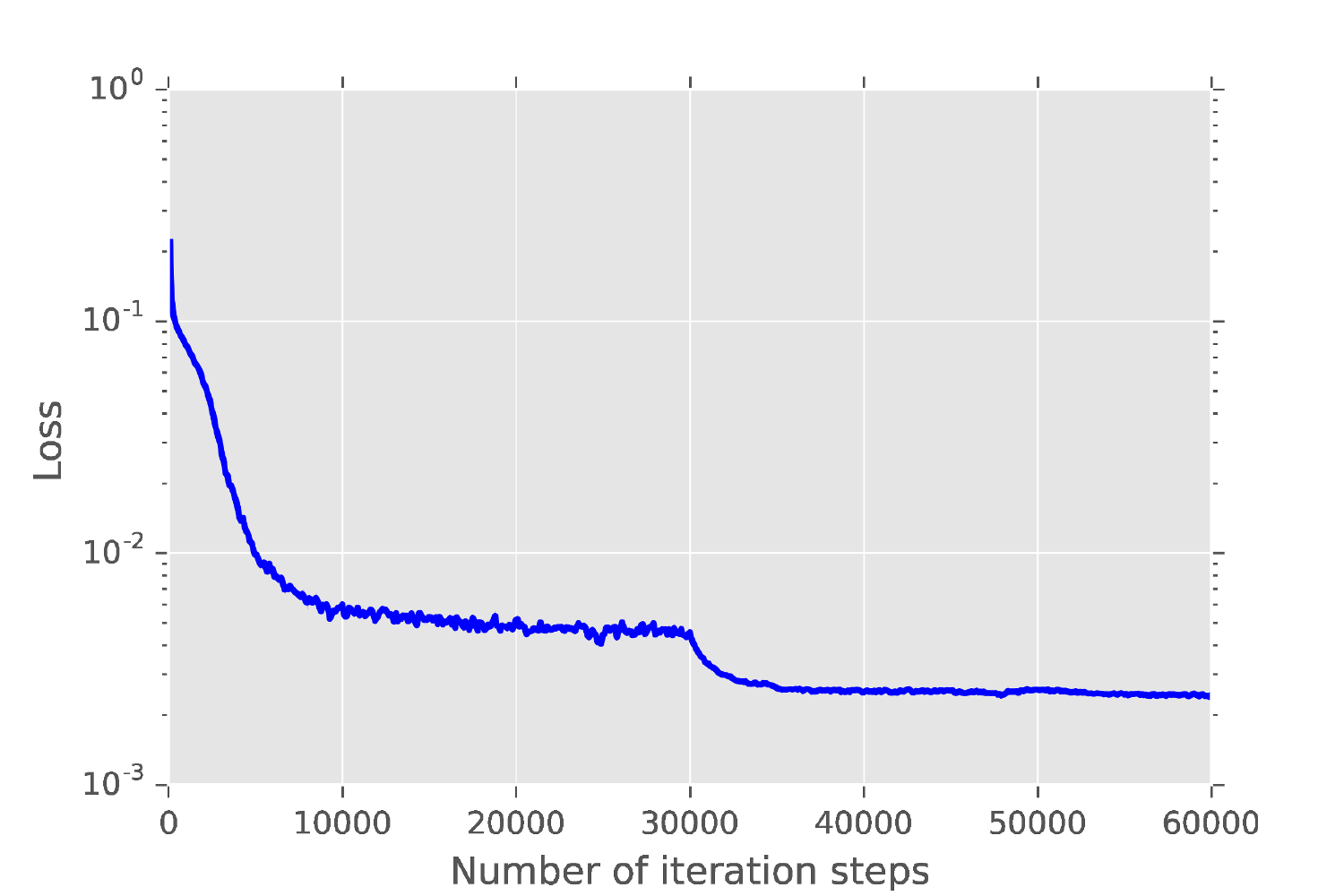}}
\caption{Relative $ L^1 $-approximation error 
of $ \mathcal{U}^{ \Theta_m } $
and mean of the loss function against $ m \in \{ 1, 2, 3, \dots, 60 \, 000 \} $
in the case of the PDE~\eqref{eq:PDE_JF} with \eqref{eq:parameter_JF}.
The deep BSDE approximation $ \mathcal{U}^{ \Theta_{ 60 \, 000 } } \approx u(0,\xi) $ 
achieves a relative $ L^1 $-approximation error 
of size $ 0.0073 $ in a runtime of $ 20 \, 389 $ seconds.
\label{fig:JF}}
\end{figure}

\begin{figure}[ht]
\centering
\setcounter{subfigure}{0}
\subfigure[Relative $ L^1 $-approximation error]{\includegraphics[width=8cm]{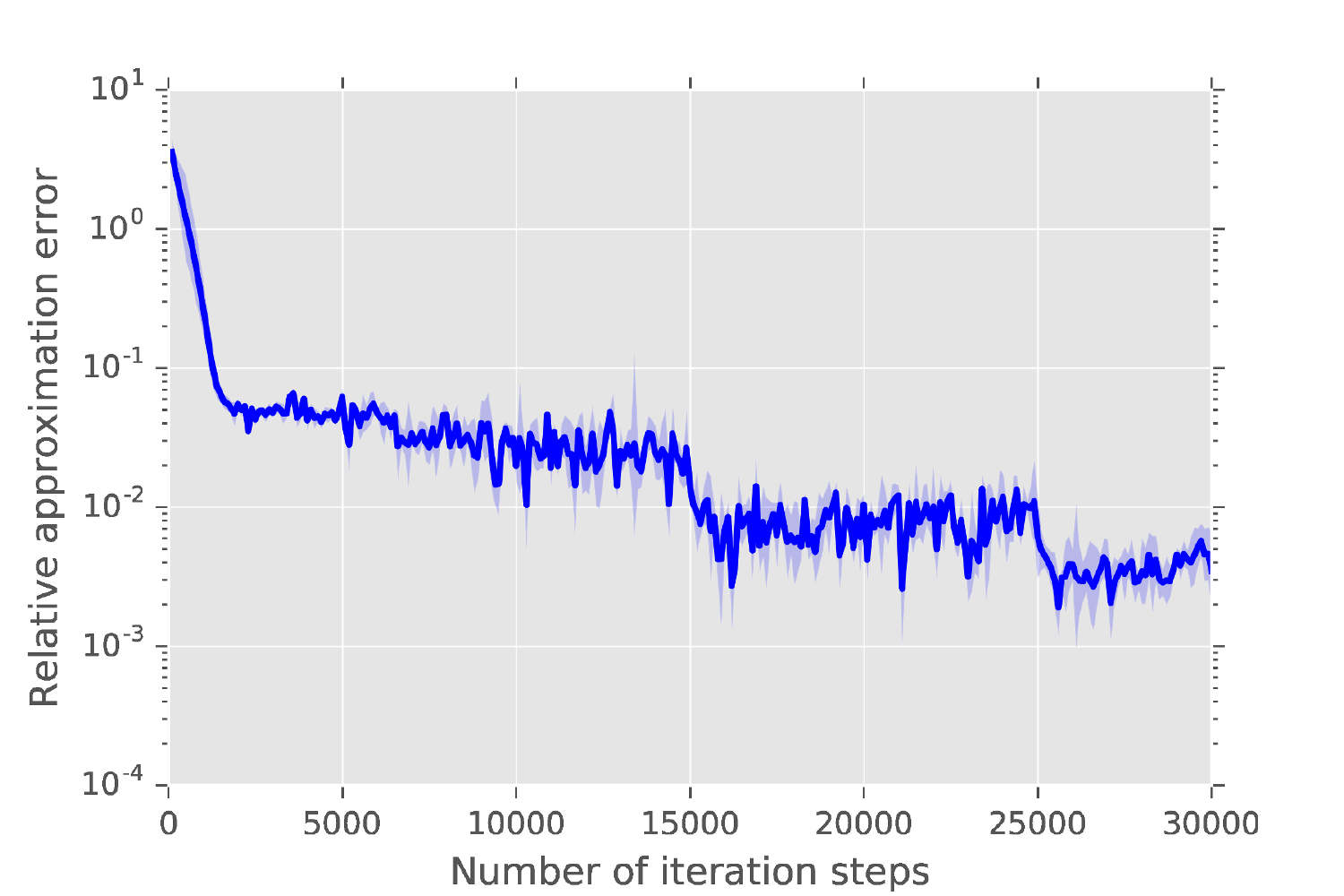}}
\subfigure[Mean of the loss function]{\includegraphics[width=8cm]{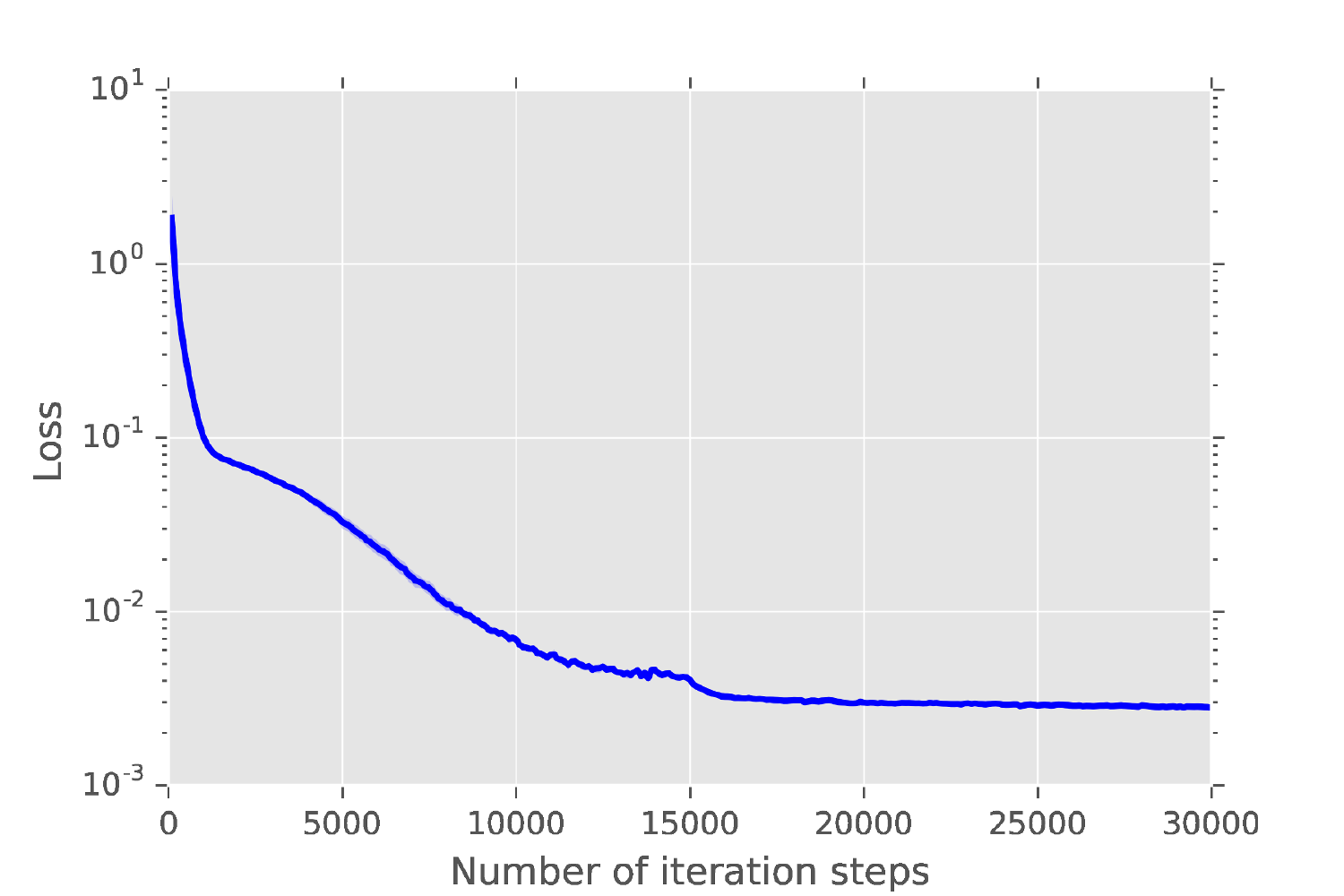}}
\caption{Relative $ L^1 $-approximation error 
of $ \mathcal{U}^{ \Theta_m } $
and mean of the loss function against $ m \in \{ 1, 2, 3, \dots, 30 \, 000 \} $
in the case of the PDE~\eqref{eq:PDE_JF} with \eqref{eq:parameter_JFb}.
The deep BSDE approximation $ \mathcal{U}^{ \Theta_{ 30 \, 000 } } \approx u(0,\xi) $ 
achieves a relative $ L^1 $-approximation error 
of size $ 0.0035 $ in a runtime of $ 4281 $ seconds.
\label{fig:JFb}}
\end{figure}

\begin{lemma}[Cf.\ Example~4.6 in Subsection~4.2 in \cite{MR3284573}]
\label{lem:JF}
Let 
$ \alpha, \kappa, T \in (0,\infty) $,
$ d \in \N $, 
let 
$ u \colon [0,T] \times \R^d \to \R $ 
be the function which satisfies 
for all 
$ t \in [0,T] $, 
$ x = (x_1, \ldots, x_d ) \in \R^d $ 
that
\begin{equation}
  u(t,x)
  =
  1 - 
  \frac{ 1 }{ 
    ( 1 + \exp( t + \kappa \sum_{ i = 1 }^d x_i ) ) 
  }
  =
  \frac{ 
    \exp( t + \kappa \sum_{ i = 1 }^d x_i )  
  }{ 
    ( 1 + \exp( t + \kappa \sum_{ i = 1 }^d x_i ) ) 
  }
  ,
\end{equation}
and let 
$ 
  f \colon [0,T] \times \R^d \times \R^{ 1 + d } \to \R
$
be the function 
which satisfies for all $ t \in [0,T] $, $ x \in \R^d $,
$ y \in \R $, $ z = ( z_1, \dots, z_d) \in \R^d $
that
\begin{equation}
  f(t,x,y,z) 
  = 
  \left(
    \alpha \kappa y 
    - 
    \frac{ 1 }{ d \kappa } 
    -
    \frac{ \alpha \kappa }{ 2 } 
  \right)
  \left(
    \smallsum\limits_{ i = 1 }^d z_i
  \right)
  .
\end{equation}
Then it holds for all 
$ t \in [0,T] $, $ x \in \R^d $ 
that
\begin{equation}
  \frac{ \partial u}{ \partial t } ( t, x )
  +
  \frac{ \alpha }{ 2 } 
  ( \Delta_x u )(t,x)
  +
  f\big( 
    t, x, u(t,x), ( \nabla_x u )(t,x) 
  \big) 
  = 0 .
\end{equation}
\end{lemma}

\begin{proof}[Proof 
of Lemma~\ref{lem:JF}]
Throughout this proof 
let 
$ \beta, \gamma \in (0,\infty) $
be the real numbers given by
\begin{equation}
  \beta = \alpha \kappa
  \qquad 
  \text{and}
  \qquad 
  \gamma = \frac{ 1 }{ d \kappa } + \frac{ \alpha \kappa }{ 2 }
\end{equation}
and let 
$ w \colon [0,T] \times \R^d \to (0,\infty) $
be the function 
which satisfies for all 
$ t \in [0,T] $, 
$ x = ( x_1 , \dots, x_d ) \in \R^d $
that
\begin{equation}
  w(t,x)
  =
  \exp\!\left(
    t
    +
    \kappa
    \sum_{ i = 1 }^d
    x_i
  \right)
  .
\end{equation}
Observe that 
for all $ t \in [0,T] $,
$ x = ( x_1, \dots, x_d ) \in \R^d $,
$ i \in \{ 1, 2, \dots, d \} $
it holds that
\begin{equation}
\label{eq:def_u}
  u(t,x)
  =
  1 
  -
  \left[ 1 + w(t,x) \right]^{ - 1 }  
  =
  \frac{
    \left[ 1 + w(t,x) \right]
  }{
    \left[ 1 + w(t,x) \right]
  }
  -
  \frac{ 1 }{
    \left[ 1 + w(t,x) \right]
  }
  =
  \frac{ w(t,x) }{ 1 + w(t,x) }
  ,
\end{equation}
\begin{equation}
\label{eq:u_time_derivative}
   \frac{ \partial u}{ \partial t } ( t, x )
  =
  \left[ 1 + w(t,x) \right]^{ - 2 }  
  \cdot 
   \frac{ \partial w}{ \partial t } ( t, x )
  =
  \frac{
    w(t,x)
  }{
    \left[ 1 + w(t,x) \right]^2
  }
  ,
\end{equation}
and
\begin{equation}
\label{eq:u_xi}
   \frac{ \partial u}{ \partial x_i } ( t, x )
  =
  \left[ 1 + w(t,x) \right]^{ - 2 }  
  \cdot 
  \frac{ \partial w}{ \partial x_i } ( t, x )
  =
  \kappa 
  \,
  w(t,x)
  \left[ 1 + w(t,x) \right]^{ - 2 } 
  .
\end{equation}
Note that 
\eqref{eq:def_u},
\eqref{eq:u_time_derivative},
and \eqref{eq:u_xi}
ensure that 
for all $ t \in [0,T] $, $ x \in \R^d $
it holds that
\begin{equation}
\label{eq:u_whole_equation}
\begin{split}
&
  \frac{ \partial u}{ \partial t } (t,x)
  +
  \frac{ \alpha }{ 2 } 
  ( \Delta_x u )(t,x)
  +
  f\big(
    t, x, u(t,x), ( \nabla_x u )(t,x) 
  \big) 
\\ &
  =
  \frac{ \partial u}{ \partial t } (t,x)
  +
  \frac{ \alpha }{ 2 } 
  ( \Delta_x u )(t,x)
  +
  \left(
    \beta u(t,x)
    -
    \gamma
  \right)
  \left(
    \smallsum\limits_{ i = 1 }^d
    \displaystyle
    \frac{ \partial u}{ \partial x_i } (t,x)
  \right)
\\ & 
  =
   \frac{ \partial u}{ \partial t } (t,x)
  +
  \frac{ \alpha }{ 2 } 
  ( \Delta_x u )(t,x)
  +
  d
  \,
   \frac{ \partial u}{ \partial x_1 } (t,x)
  \big(
    \beta u(t,x)
    -
    \gamma
  \big)
\\ & 
  =
  \frac{
    w(t,x)
  }{
    \left[ 1 + w(t,x) \right]^2
  }
  +
  \frac{ \alpha }{ 2 } 
  ( \Delta_x u )(t,x)
  +
  \frac{
    d \kappa w(t,x)
  }{
    \left[ 1 + w(t,x) \right]^2
  }
    \left(
      \frac{ \beta w(t,x) }{
        \left[ 1 + w(t,x) \right]
      }
      -
      \gamma
    \right)
  .
\end{split}
\end{equation}
Moreover, observe that 
\eqref{eq:u_xi}
demonstrates that
for all $ t \in [0,T] $,
$ x \in \R^d $
it holds that
\begin{equation}
\label{eq:uxi2}
\begin{split}
   \frac{ \partial^2 u}{ \partial x_i^2 } ( t, x )
& =
  \kappa
  \,
   \frac{ \partial w}{ \partial x_i } ( t, x )
  \left[ 1 + w(t,x) \right]^{ - 2 }  
  -
  2 \kappa w(t,x)
  \left[ 1 + w(t,x) \right]^{ - 3 }  
   \frac{ \partial w}{ \partial x_i } ( t, x )
\\ &
  =
  \frac{ \kappa^2 w(t,x) }{
    \left[ 1 + w(t,x) \right]^2
  }
  -
  \frac{ 2 \kappa^2 | w(t,x) |^2 }{ 
    \left[ 1 + w(t,x) \right]^3
  }
  =
  \frac{
    \kappa^2 w(t,x)
  }{
    \left[ 1 + w(t,x) \right]^2
  }
  \left[ 
    1
    -
    \frac{ 2 w(t,x) }{ 
      \left[ 
        1 + w(t,x)
      \right]
    }
  \right]
  .
\end{split}
\end{equation}
Hence, we obtain that for all 
$ t \in [0,T] $, 
$ x \in \R^d $
it holds that
\begin{equation}
\label{eq:delta_u_x}
  \frac{ \alpha }{ 2 }
  ( 
    \Delta_x u
  )( t, x )
  =
  \frac{ d \alpha }{ 2 }
  \frac{ \partial^2 u }{ \partial x_1^2 }( t, x )
  =
  \frac{
    d \alpha \kappa^2 w(t,x)
  }{
    2 \left[ 1 + w(t,x) \right]^2
  }
  \left[ 
    1
    -
    \frac{ 2 w(t,x) }{ 
      \left[ 
        1 + w(t,x)
      \right]
    }
  \right]
  .
\end{equation}
Combining this with \eqref{eq:u_whole_equation} implies that
for all $ t \in [0,T] $, $ x \in \R^d $
it holds that
\begin{equation}
\begin{split}
&
  \frac{ \partial u}{ \partial t } (t,x)
  +
  \frac{ \alpha }{ 2 } 
  ( \Delta_x u )(t,x)
  +
  f\big( 
    t, x, u(t,x), ( \nabla_x u )(t,x) 
  \big) 
\\ &
  =
  \frac{
    w(t,x)
    \,
    [ 
      1 
      -
      d \kappa \gamma
    ]
  }{
    \left[ 1 + w(t,x) \right]^2
  }
  +
  \frac{ \alpha }{ 2 } 
  ( \Delta_x u )(t,x)
  +
  \frac{
    d \beta \kappa | w(t,x) |^2
  }{
    \left[ 1 + w(t,x) \right]^3
  }
\\ &
  =
  \frac{
    w(t,x)
    \,
    [ 
      1 
      -
      d \kappa \gamma 
    ]
  }{
    \left[ 1 + w(t,x) \right]^2
  }
  +
  \frac{
    d \alpha \kappa^2 w(t,x)
  }{
    2 \left[ 1 + w(t,x) \right]^2
  }
  \left[ 
    1
    -
    \frac{ 2 w(t,x) }{ 
      \left[ 
        1 + w(t,x)
      \right]
    }
  \right]
  +
  \frac{
    d \beta \kappa | w(t,x) |^2
  }{
    \left[ 1 + w(t,x) \right]^3
  }
\\ &
  =
  \frac{
    w(t,x)
    \big[ 
      1 
      -
      d \kappa \gamma 
      +
      \frac{ d \alpha \kappa^2 }{ 2 }
    \big]
  }{
    \left[ 1 + w(t,x) \right]^2
  }
  -
  \frac{
    d \alpha \kappa^2 | w(t,x) |^2
  }{
    \left[ 1 + w(t,x) \right]^3
  }
  +
  \frac{
    d \beta \kappa | w(t,x) |^2
  }{
    \left[ 1 + w(t,x) \right]^3
  }
  .
\end{split}
\end{equation}
The fact that 
$
  \alpha \kappa^2 = \beta \kappa
$
hence demonstrates that
for all $ t \in [0,T] $, $ x \in \R^d $
it holds that
\begin{equation}
\begin{split}
&
   \frac{ \partial u}{ \partial t } (t,x)
  +
  \frac{ \alpha }{ 2 } 
  ( \Delta_x u )(t,x)
  +
  f\big( 
    t, x, u(t,x), ( \nabla_x u )(t,x) 
  \big) 
\\ &
  =
  \frac{
    w(t,x)
    \big[ 
      1 
      -
      d \kappa \gamma 
      +
      \frac{ d \alpha \kappa^2 }{ 2 }
    \big]
  }{
    \left[ 1 + w(t,x) \right]^2
  }
  .
\end{split}
\end{equation}
This and the fact that
$
  1 
  +
  \frac{ d \alpha \kappa^2 }{ 2 }
  =
  d \kappa \gamma
$
show that
for all $ t \in [0,T] $, $ x \in \R^d $
it holds that
\begin{equation}
\begin{split}
&
   \frac{ \partial u}{ \partial t } (t,x)
  +
  \frac{ \alpha }{ 2 } 
  ( \Delta_x u )(t,x)
  +
  f\big( 
    t, x, u(t,x), ( \nabla_x u )(t,x) 
  \big) 
=
  0
  .
\end{split}
\end{equation}
The proof of Lemma~\ref{lem:JF}
is thus completed.
\end{proof}

\subsection{An example PDE with quadratically growing derivatives and an explicit solution}
\label{sec:GobetTurke1}

In this subsection we consider a high-dimensional version 
of the example analyzed numerically in 
Gobet \& Turkedjiev~\cite[Section~5]{GobetTurkedjiev2016MathComp}.
More specifically, Gobet \& Turkedjiev~\cite[Section~5]{GobetTurkedjiev2016MathComp} 
employ the PDE in \eqref{eq:PDE_GobetTurke1} below 
as a numerical test example but with 
the time horizont $ T = \nicefrac{ 2 }{ 10 } $
instead of $ T = 1 $ in this article and 
with the dimension $ d \in \{ 3, 5, 7 \} $
instead of $ d = 100 $ in this article.

Assume the setting in Subsection~\ref{sec:example_setting}, 
let 
$ 
  \alpha = \nicefrac{ 4 }{ 10 } 
$,
let 
$ \psi \colon [0,T] \times \R^d \to \R $ 
be the function which satisfies 
for all $ (t,x) \in [0,T] \times \R^d $
that
$
  \psi(t,x)
  =
  \sin\!\left(
    [ T - t + \| x \|^2_{ \R^d } ]^{ \alpha }
  \right)
$,
and assume 
for all
$
  s \in [0,T]
$, 
$
  t \in [0,T) 
$,
$
  x, w
  \in \R^d
$, 
$
  y \in \R
$, 
$
  z \in \R^{ 1 \times d } 
$,
$ m \in \N $
that 
$ T = 1 $,
$ d = 100 $,
$ N = 30 $,
$
  \gamma_m 
  = 
  5 \cdot 10^{ -3 }
  =
  \frac{ 5 }{ 1000 }
  = 0.005
$,
$
  \mu( t, x ) = 0
$,
$
  \sigma( t, x ) w = w
$,
$
  \xi = ( 0, 0, \dots, 0 ) \in \R^d
$,
$
  \Upsilon( t, s, x, w )
  =
  x + w
$,
$
  g(x) 
  =   
  \sin(
    \| x \|_{ \R^d }^{ 2 \alpha }
  )
$,
and
\begin{equation}
  f(t,x,y,z) 
  = 
  \| z \|^2_{ \R^{ 1 \times d } } 
  - 
  \| ( \nabla_x \psi )( t, x ) \|^2_{ \R^d }
  -
   \frac{ \partial \psi}{ \partial t } (t,x)
  -
  \frac{ 1 }{ 2 } \, 
  ( \Delta_x \psi )(t,x)
  .
\end{equation}
Note that the solution $ u $ of the 
PDE~\eqref{eq:PDE_numerics}
then satisfies for all 
$ 
  t \in [0,T) 
$, 
$ 
  x = (x_1, x_2, \ldots, x_d) \in \R^d 
$
that
$
  u(T,x) = g(x) 
$
and
\begin{equation}  
\label{eq:PDE_GobetTurke1}
   \frac{ \partial u}{ \partial t } ( t, x )
  + 
  \| ( \nabla_x u )( t, x ) \|^2_{ \R^d } 
  +
  \frac{ 1 }{ 2 } \, 
  ( \Delta_x u )(t,x)
  = 
   \frac{ \partial \psi}{ \partial t } (t,x)
  + 
  \| ( \nabla_x \psi )( t, x ) \|^2_{ \R^d }
  +
  \frac{ 1 }{ 2 } \, 
  ( \Delta_x \psi )(t,x)
  .
\end{equation}
On the left hand side of Figure~\ref{fig:GobetTurke1}
we present approximatively 
the relative $ L^1 $-approximatin error
associated to 
$ \mathcal{U}^{ \Theta_m } $
against 
$ m \in \{ 1, 2, 3, \dots, 4000 \} $
based on $ 5 $ independent realizations
($ 5 $ independent runs).
On the right hand side of Figure~\ref{fig:GobetTurke1}
we present approximatively 
the mean of the loss function 
associated to 
$ \Theta_m $ 
against 
$ m \in \{ 1, 2, 3, \dots, 4000 \} $
based on $ 256 $ Monte Carlo samples 
and $ 5 $ independent realizations ($ 5 $ independent runs).
\begin{figure}[ht]
\centering
\setcounter{subfigure}{0}
\subfigure[Relative $ L^1 $-approximation error]{\includegraphics[width=8cm]{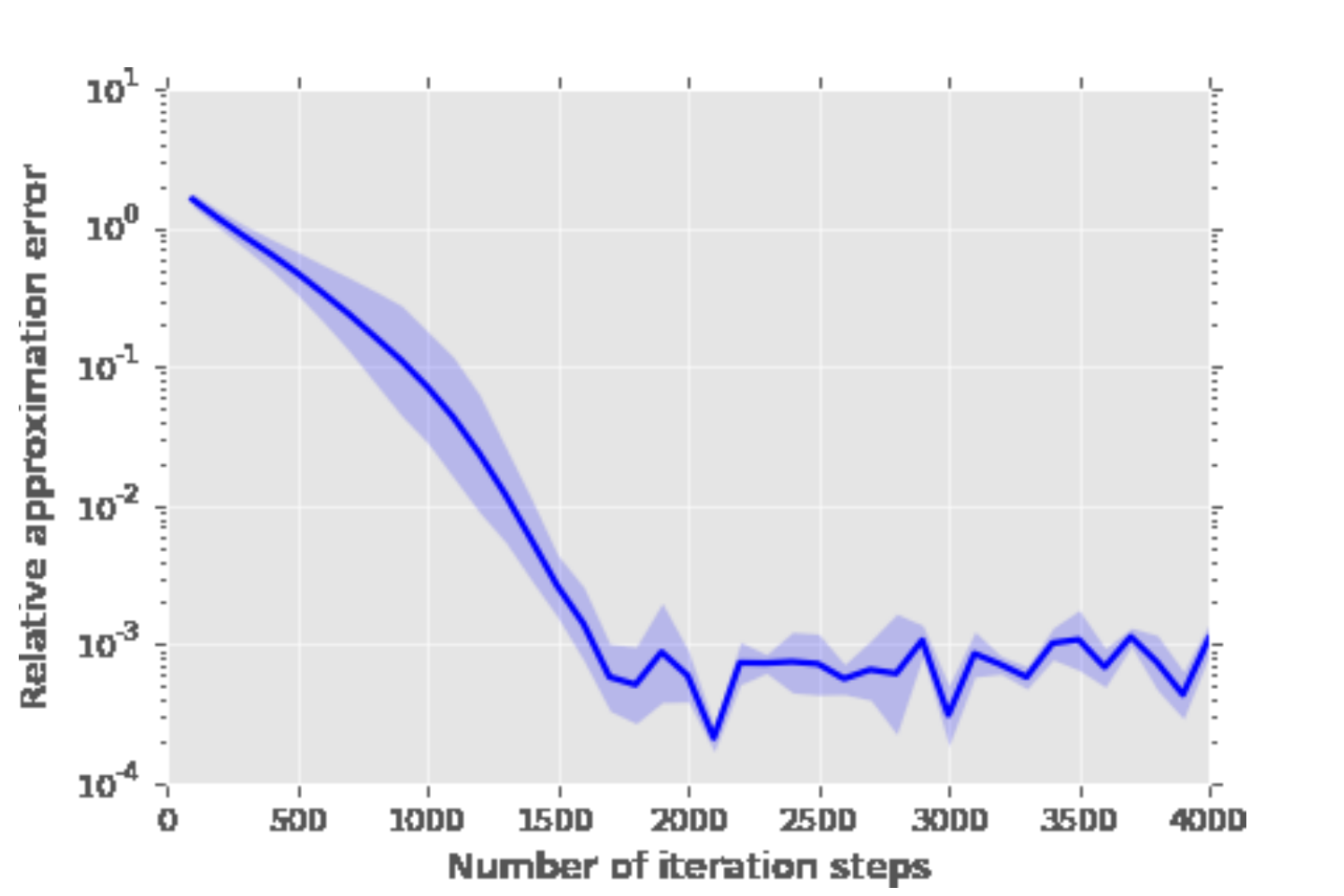}}
\subfigure[Mean of the loss function]{\includegraphics[width=8cm]{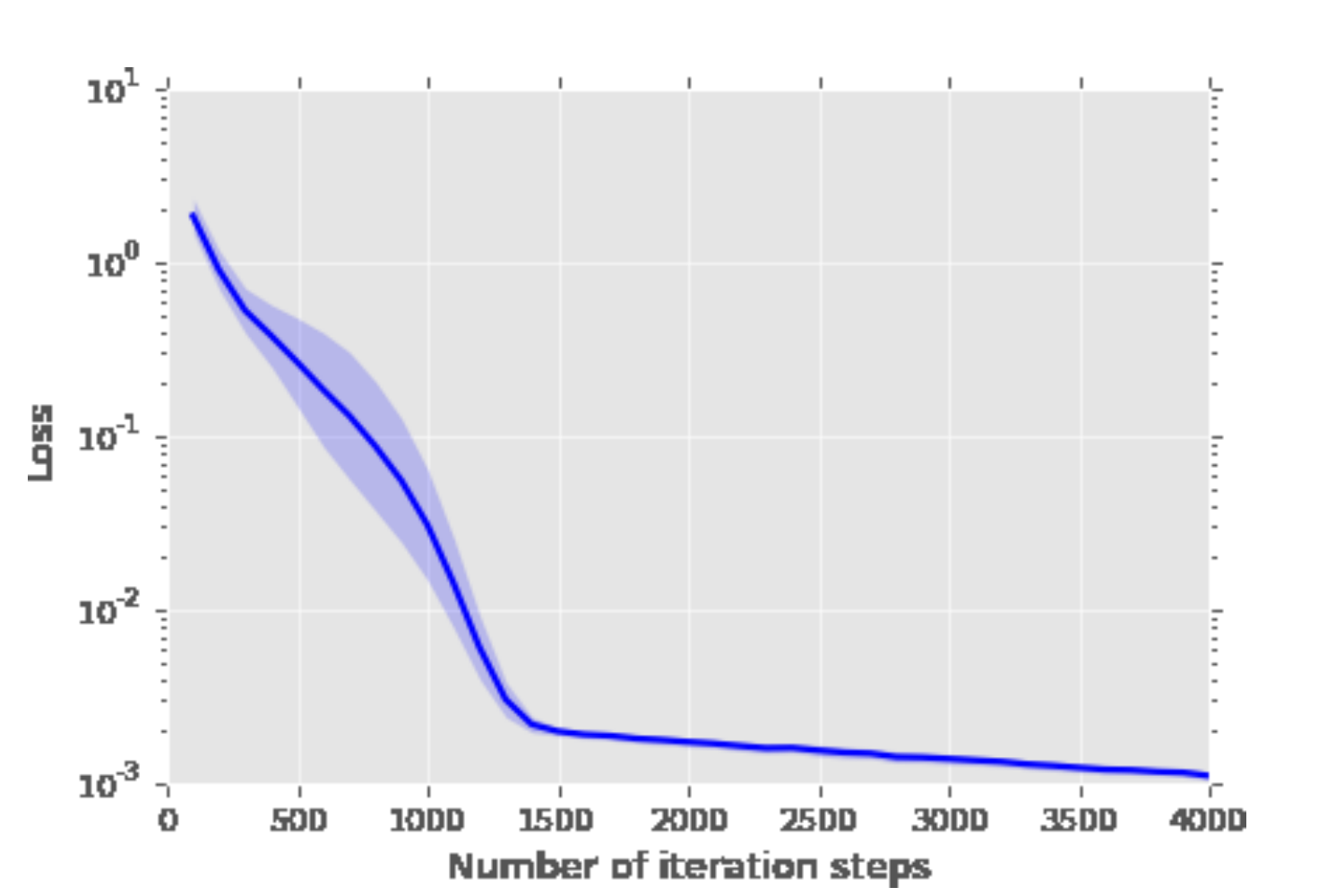}}
\caption{Relative $ L^1 $-approximation error 
of $ \mathcal{U}^{ \Theta_m } $
and mean of the loss function against $ m \in \{ 1, 2, 3, \dots, 4000 \} $
in the case of the PDE~\eqref{eq:PDE_GobetTurke1}.
The deep BSDE approximation $ \mathcal{U}^{ \Theta_{ 4 000 } } \approx u(0,\xi) $ 
achieves a relative $ L^1 $-approximation error 
of size $ 0.0009 $ in a runtime of $ 957 $ seconds.
\label{fig:GobetTurke1}}
\end{figure}

\subsection{Time-dependent reaction-diffusion-type example PDEs with oscillating explicit solutions}
\label{sec:GobetTurke2}

In this subsection we consider a high-dimensional version 
of the example PDE analyzed numerically 
in Gobet \& Turkedjiev~\cite[Subsection~6.1]{GobetTurkedjiev2016SPA}.
More specifically,
Gobet \& Turkedjiev~\cite[Subsection~6.1]{GobetTurkedjiev2016SPA} 
employ the PDE in \eqref{eq:PDE_GobetTurke2} below 
as a numerical test example but in two space-dimensions ($ d = 2 $)
instead of in hundred space-dimensions ($ d = 100 $) as in this article.

Assume the setting in Subsection~\ref{sec:example_setting}, 
let 
$ 
  \kappa = \nicefrac{ 6 }{ 10 } 
$, 
$ 
  \lambda = \nicefrac{ 1 }{ \sqrt{d} }
$,
assume 
for all
$
  s, t \in [0,T]
$, 
$
  x = ( x_1, \dots, x_d ),
  w = ( w_1, \dots, w_d ) 
  \in \R^d
$, 
$
  y \in \R
$, 
$
  z \in \R^{ 1 \times d } 
$,
$ m \in \N $
that 
$
  \gamma_m = \frac{ 1 }{ 100 } = 0.01
$,
$ T = 1 $,
$ d = 100 $,
$ N = 30 $,
$
  \mu( t, x ) = 0
$,
$
  \sigma( t, x ) w = w
$,
$
  \xi = ( 0, 0, \dots, 0 ) \in \R^d
$,
$
  \Upsilon( s, t, x, w )
  =
  x + w
$,
$
  g(x) = 1 + \kappa + \sin( \lambda \sum_{ i = 1 }^d x_i )
$,
and
\begin{equation}
\label{eq:def_f_GobetTurke2}
  f(t, x, y, z) 
  = 
  \min\!\Big\{ 
    1 ,
    \big[
      y 
      - 
      \kappa 
      - 
      1
      - 
      \sin\!\big( 
        \textstyle
        \lambda \sum_{ i = 1 }^d x_i 
      \big) 
      \exp\!\big( 
        \tfrac{ 
          \lambda^2 d ( t - T ) 
        }{ 2 } 
      \big)
    \big]^2
  \Big\}
  .
\end{equation}
Note that the solution $ u $ of the 
PDE~\eqref{eq:PDE_numerics}
then satisfies for all 
$ 
  t \in [0,T) 
$, 
$ 
  x = (x_1, x_2, \ldots, x_d) \in \R^d 
$
that
$
  u(T,x) = g(x) 
$
and
\begin{equation}  
\label{eq:PDE_GobetTurke2}
  \frac{ \partial u}{ \partial t } ( t, x )
  + 
  \min\!\Big\{ 
    1 ,
    \big[
      u(t,x)
      - 
      \kappa 
      - 
      1
      - 
      \sin\!\big( 
        \textstyle
        \lambda \sum_{ i = 1 }^d x_i 
      \big) 
      \exp\!\big( 
        \frac{ 
          \lambda^2 d ( t - T ) 
        }{ 2 } 
      \big)
    \big]^2
  \Big\}
  +
  \frac{ 1 }{ 2 } \, 
  ( \Delta_x u )(t,x)
  = 0
\end{equation}
(cf.\ Lemma~\ref{lem:GobetTurke2} below).
On the left hand side of Figure~\ref{fig:GobetTurke1}
we present approximatively 
the relative $ L^1 $-approximatin error
associated to 
$ \mathcal{U}^{ \Theta_m } $
against 
$ m \in \{ 1, 2, 3, \dots, 24 000 \} $
based on $ 5 $ independent realizations
($ 5 $ independent runs).
On the right hand side of Figure~\ref{fig:GobetTurke1}
we present approximatively 
the mean of the loss function 
associated to 
$ \Theta_m $ 
against 
$ m \in \{ 1, 2, 3, \dots, 24 000 \} $
based on $ 256 $ Monte Carlo samples 
and $ 5 $ independent realizations ($ 5 $ independent runs).
\begin{figure}[ht]
\centering
\setcounter{subfigure}{0}
\subfigure[Relative $ L^1 $-approximation error]{\includegraphics[width=8cm]{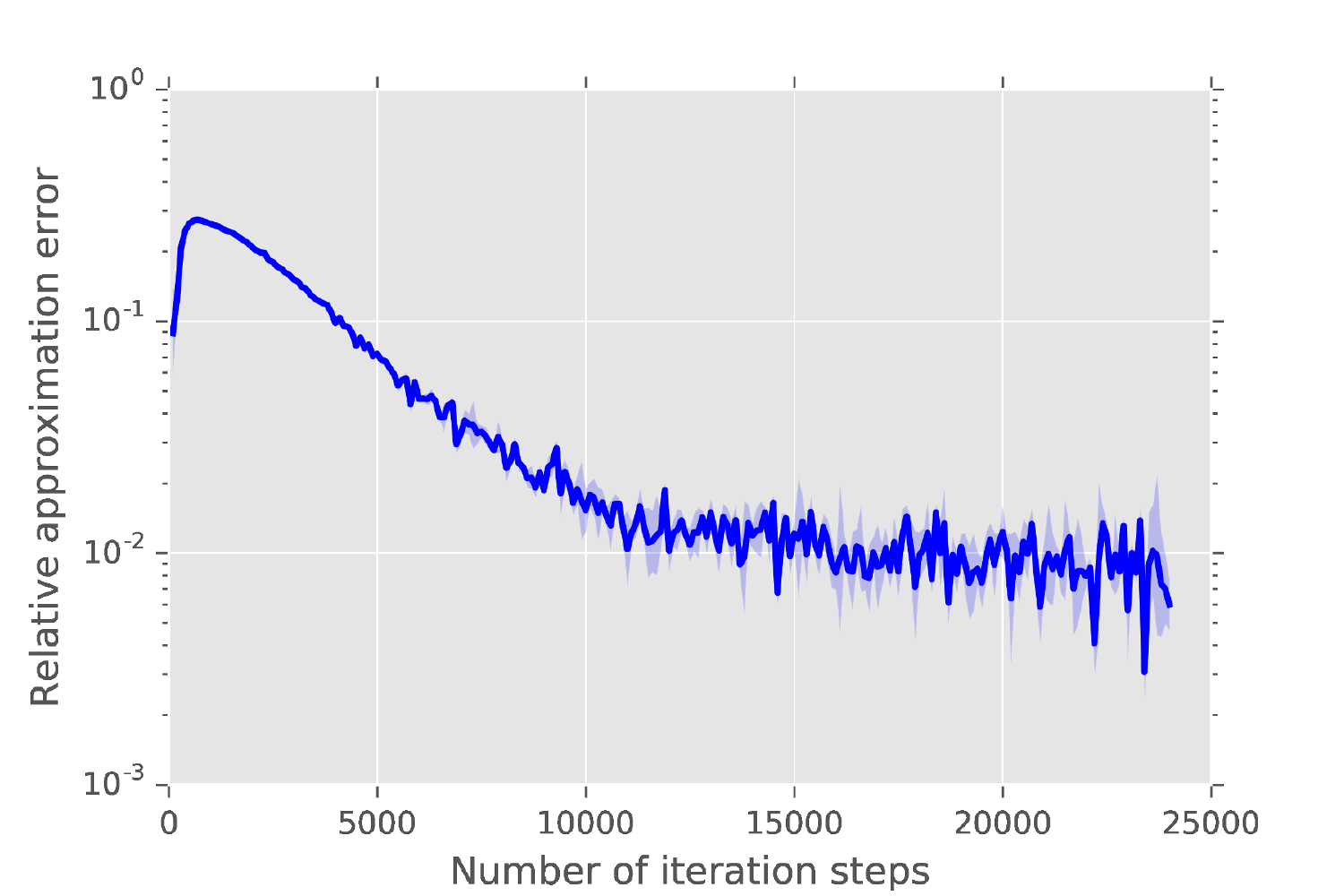}}
\subfigure[Mean of the loss function]{\includegraphics[width=8cm]{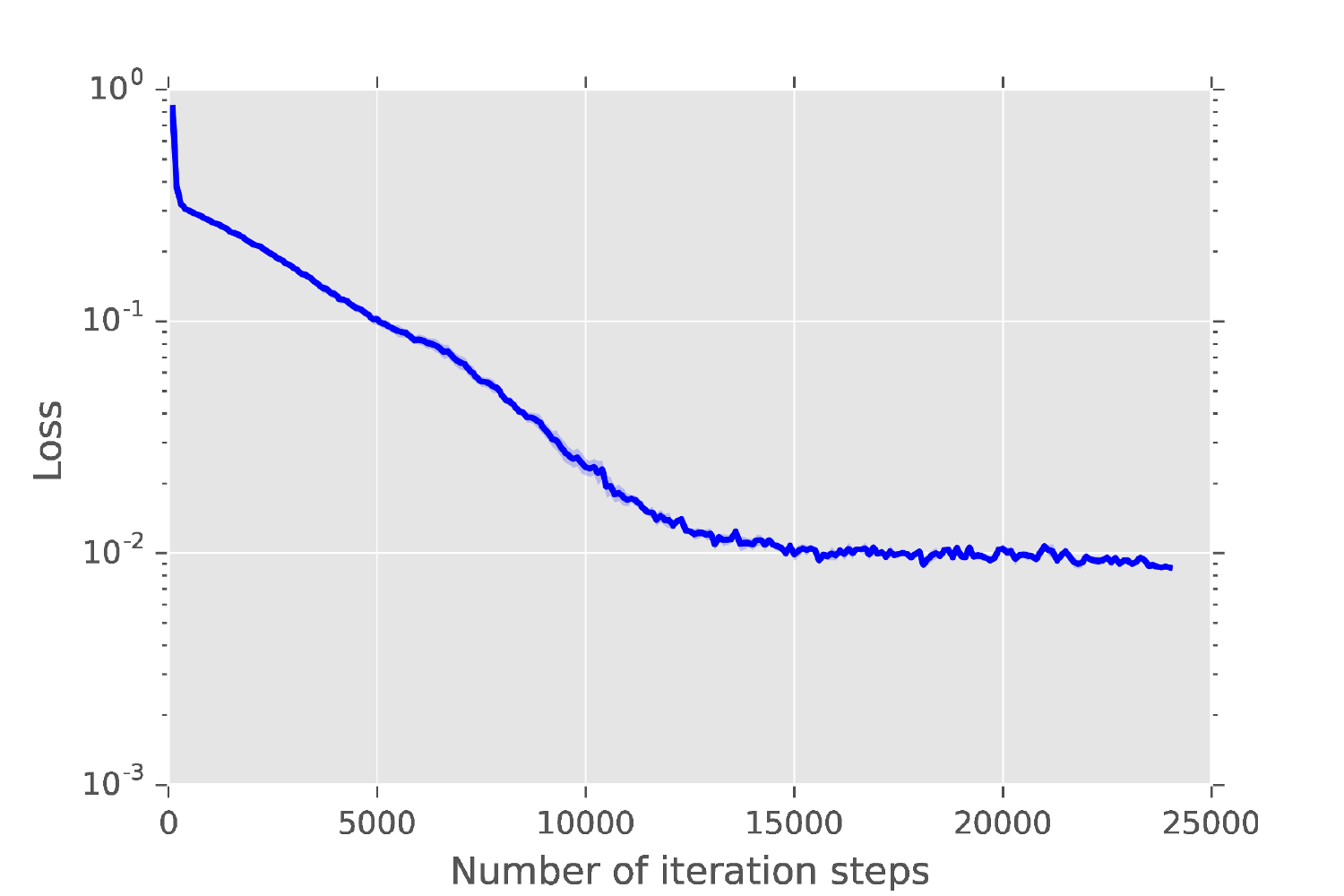}}
\caption{Relative $ L^1 $-approximation error 
of $ \mathcal{U}^{ \Theta_m } $
and mean of the loss function against $ m \in \{ 1, 2, 3, \dots, 24 000 \} $
in the case of the PDE~\eqref{eq:PDE_GobetTurke2}.
The deep BSDE approximation $ \mathcal{U}^{ \Theta_{ 24 000 } } \approx u(0,\xi) $ achieves a relative $ L^1 $-approximation error 
of size $ 0.0060 $ in a runtime of $ 4872 $ seconds.
\label{fig:GobetTurke2}}
\end{figure}
\begin{lemma}[Cf.\ Subsection~6.1 in \cite{GobetTurkedjiev2016SPA}]
\label{lem:GobetTurke2}
Let
$ T, \kappa, \lambda \in (0,\infty) $,
$ d \in \N $
and let 
$ u \colon [0,T] \times \R^d \to \R $
be the function which satisfies
for all $ t \in [0,T] $, $ x = ( x_1, \dots, x_d ) \in \R^d $
that
\begin{equation}
  u(t,x) 
  = 
  1
  +
  \kappa 
  +
  \sin\!\big( 
    \textstyle
    \lambda \sum_{ i = 1 }^d x_i 
  \big) 
  \exp\!\big( 
    \tfrac{ 
      \lambda^2 d ( t - T ) 
    }{ 2 } 
  \big)
  .
\end{equation}
Then it holds 
for all 
$ t \in [0,T] $, $ x = ( x_1, \dots, x_d ) \in \R^d $
that
$
  u \in C^{ 1, 2 }( [0,T] \times \R^d, \R )
$,
$
  u(T,x) = 
  1 + \kappa + \sin( \lambda \sum_{ i = 1 }^d x_i )
$,
and
\begin{equation}  
  \frac{ \partial u}{ \partial t } ( t, x )
  + 
  \min\!\Big\{ 
    1 ,
    \big[
      u(t,x)
      - 
      \kappa 
      - 
      1
      - 
      \sin\!\big( 
        \textstyle
        \lambda \sum_{ i = 1 }^d x_i 
      \big) 
      \exp\!\big( 
        \frac{ 
          \lambda^2 d ( t - T ) 
        }{ 2 } 
      \big)
    \big]^2
  \Big\}
  +
  \frac{ 1 }{ 2 } \, 
  ( \Delta_x u )(t,x)
  = 0
  .
\end{equation}
\end{lemma}

\begin{proof}[Proof of Lemma~\ref{lem:GobetTurke2}]
Note that for all $ t \in [0,T] $,
$ x = ( x_1, \dots, x_d ) \in \R^d $ it holds that
\begin{equation}
\label{eq:time_derivative_v}
\begin{split}
   \frac{ \partial u }{ \partial t } ( t, x )
& =
  \frac{ \lambda^2 d }{ 2 }
  \sin\!\big( 
    \textstyle
    \lambda \sum_{ i = 1 }^d x_i 
  \big) 
  \exp\!\big( 
    \frac{ 
      \lambda^2 d ( t - T ) 
    }{ 2 } 
  \big)
  .
\end{split}
\end{equation}
In addition, observe that
for all $ t \in [0,T] $, $ x = (x_1, \dots, x_d) \in \R^d $,
$ k \in \{ 1, 2, \dots, d \} $
it holds that
\begin{equation}
  \frac{ \partial u }{ \partial x_k } 
  ( t, x )
  =
  \lambda
  \cos\!\big( 
    \textstyle
    \lambda \sum_{ i = 1 }^d x_i 
  \big) 
  \exp\!\big( 
    \frac{ 
      \lambda^2 d ( t - T ) 
    }{ 2 } 
  \big)
  .
\end{equation}
Hence, we obtain that
for all $ t \in [0,T] $, $ x = (x_1, \dots, x_d) \in \R^d $,
$ k \in \{ 1, \dots, d \} $
it holds that
\begin{equation}
  \frac{ \partial^2 u }{ \partial x_k^2 } 
  ( t, x )
  =
  -
  \lambda^2
  \sin\!\big( 
    \textstyle
    \lambda \sum_{ i = 1 }^d x_i 
  \big) 
  \exp\!\big( 
    \frac{ 
      \lambda^2 d ( t - T ) 
    }{ 2 } 
  \big)
  .
\end{equation}
This ensures that
for all $ t \in [0,T] $, $ x = (x_1, \dots, x_d) \in \R^d $
it holds that
\begin{equation}
\begin{split}
  ( \Delta_x u )( t, x )
& =
  -
  \,
  d \,
  \lambda^2
  \sin\!\big( 
    \textstyle
    \lambda \sum_{ i = 1 }^d x_i 
  \big) 
  \exp\!\big( 
    \frac{ 
      \lambda^2 d ( t - T ) 
    }{ 2 } 
  \big)
  .
\end{split}
\end{equation}
Combining this with \eqref{eq:time_derivative_v}
proves that
for all $ t \in [0,T] $, $ x = (x_1, \dots, x_d) \in \R^d $
it holds that
\begin{equation}
\begin{split}
   \frac{ \partial u }{ \partial t } ( t, x )
  +
  \frac{ 1 }{ 2 } \,
  ( \Delta_x u )( t, x )
  = 0
  .
\end{split}
\end{equation}
This demonstrates that
for all $ t \in [0,T] $, $ x = (x_1, \dots, x_d) \in \R^d $
it holds that
\begin{equation}  
\begin{split}
&
  \frac{ \partial v}{ \partial t } ( t, x )
  + 
  \min\!\Big\{ 
    1 ,
    \big[
      v(t,x)
      - 
      \kappa 
      - 
      1
      - 
      \sin\!\big( 
        \textstyle
        \lambda \sum_{ i = 1 }^d x_i 
      \big) 
      \exp\!\big( 
        \frac{ 
          \lambda^2 d ( t - T ) 
        }{ 2 } 
      \big)
    \big]^2
  \Big\}
  +
  \frac{ 1 }{ 2 } \, 
  ( \Delta_x v )(t,x)
\\ &
  =
   \frac{ \partial v}{ \partial t } ( t, x )
  +
  \frac{ 1 }{ 2 } \, 
  ( \Delta_x v )(t,x)
  = 0
  .   
\end{split}
\end{equation}
The proof of Lemma~\ref{lem:GobetTurke2} is thus completed.
\end{proof}

\section{Appendix A: Special cases of the proposed algorithm}
\label{sec:special_cases}

In this subsection we illustrate the general 
algorithm in Subsection~\ref{sec:general_case}
in several special cases.
More specifically, 
in Subsections~\ref{sec:SGD} and \ref{sec:Adam}
we provide special choices for the 
functions $ \psi_m $, $ m \in \N $,
and $ \Psi_m $, $ m \in \N $, 
employed in \eqref{eq:theta_dynamics}
and in Subsections~\ref{sec:GBM}
and \ref{sec:EM} 
we provide special choices for the 
function $ \Upsilon $
in \eqref{eq:X_dynamic}.

\subsection{Stochastic gradient descent (SGD)}
\label{sec:SGD}

\begin{example}
Assume the setting in Subsection~\ref{sec:general_case},
let 
$ ( \gamma_m )_{ m \in \N } \subseteq (0,\infty) $,
and assume 
for all 
$ m \in \N $,
$ x \in \R^{ \varrho } $,
$
  ( \varphi_j )_{ j \in \N }
  \in ( \R^{ \rho } )^{ \N }
$
that
\begin{equation}
  \varrho 
  = 
  \rho
  ,
  \qquad
  \Psi_m( x, ( \varphi_j )_{ j \in \N } )
  =
  \varphi_1
  ,
\qquad 
  \text{and}
\qquad
  \psi_m( x ) = \gamma_m x
  .
\end{equation}
Then it holds for all $ m \in \N $ that
\begin{equation}
  \Theta_{ m }
  =
  \Theta_{ m - 1 }
  -
  \gamma_{ m }
  \Phi_{ m - 1, 1 }( \Theta_{ m - 1 } )
  .
\end{equation}
\end{example}

\subsection{Adaptive Moment Estimation (Adam) with mini-batches}
\label{sec:Adam}

In this subsection we illustrate how the so-called Adam optimizer (see \cite{Kingma2015}) 
can be employed in conjunction with the deep BSDE solver 
in Subsection~\ref{sec:general_case} (cf.\ also Subsection~\ref{sec:example_setting} above).

\begin{example}
\label{ex:Adam}
Assume the setting in Subsection~\ref{sec:general_case},
assume that 
$
  \varrho   
  = 
  2 \rho
$,
let 
$
  \operatorname{Pow}_r \colon
  \R^{ \rho }
  \to 
  \R^{ \rho }
$,
$ r \in (0,\infty) $,
be the functions 
which satisfy
for all 
$ r \in (0,\infty) $,
$ 
  x = ( x_1, \dots, x_{ \rho } ) \in \R^{ \rho }
$ 
that
\begin{equation}
  \operatorname{Pow}_{ r }( x )
  = ( | x_1 |^r, \dots, | x_{ \rho } |^r )
  ,
\end{equation}
let 
$ \varepsilon \in (0,\infty) $,
$ ( \gamma_m )_{ m \in \N } \subseteq (0,\infty) $,
$ ( J_m )_{ m \in \N_0 } \subseteq \N $,
$
  \mathbb{X}, \mathbb{Y} \in (0,1)
$,
let 
$
  {\bf m}
  ,
  \mathbb{M}
  \colon
  \N_0 \times \Omega \to \R^{ \rho }
$
be the stochastic processes
which satisfy for all 
$ m \in \N_0 $
that
$
  \Xi_m = ( {\bf m}_m , \mathbb{M}_m )
$,
and assume 
for all 
$ m \in \N $,
$ x, y \in \R^{ \rho } $,
$
  ( \varphi_j )_{ j \in \N }
  \in ( \R^{ \rho } )^{ \N }
$
that
\begin{equation}
  \Psi_m( x, y, ( \varphi_j )_{ j \in \N } )
  =
  \big(
  \mathbb{X} x
  +
    ( 1 - \mathbb{X} ) 
    \big(
      \tfrac{
        1 
      }{ 
        J_m 
      }
      \smallsum_{ j = 1 }^{ J_m }
      \varphi_j
    \big)
  ,
  \mathbb{Y} y
  +
    ( 1 - \mathbb{Y} ) 
    \operatorname{Pow}_2\!\big(
      \frac{ 1 }{ J_m }
      \smallsum_{ j = 1 }^{ J_m }
      \varphi_j
    \big)
  \big)
\end{equation}
and
\begin{equation}
  \psi_m( x, y ) 
  = 
  \left[ 
    \varepsilon
    +
    \operatorname{Pow}_{ \nicefrac{ 1 }{ 2 } }\!\left(
      \frac{
        y
      }{
        ( 1 - \mathbb{Y}^m )
      }
    \right)
  \right]^{ - 1 }
  \frac{ 
    \gamma_m x 
  }{ 
    ( 1 - \mathbb{X}^m ) 
  }
  .
\end{equation}
Then it holds for all $ m \in \N $ that
\begin{equation}
\begin{split}
  \Theta_{ m }
& =
  \Theta_{ m - 1 }
  -
  \left[ 
    \varepsilon
    +
    \operatorname{Pow}_{ 
      \nicefrac{ 1 }{ 2 } 
    }\!\left(
      \frac{
        \mathbb{M}_{ m }
      }{
        ( 1 - \mathbb{Y}^{ m } )
      }
    \right)
  \right]^{ - 1 }
  \frac{
    \gamma_{ m } 
    {\bf m}_{ m }
  }{
    ( 1 - \mathbb{X}^{ m } )
  }
  ,
\\
  {\bf m}_m 
& =
  \mathbb{X}
  \,
  {\bf m}_{ m - 1 }
  +
  \frac{ 
    ( 1 - \mathbb{X} )
  }{ J_m }
  \left(
  \sum_{ j = 1 }^{ J_m }
      \Phi^{ m - 1 , j }_{
        \mathbb{S}_m
      }( 
        \Theta_{ m - 1 } 
      ) 
  \right)
  ,
\\
  \mathbb{M}_m 
&
  =
  \mathbb{Y}
  \,
  \mathbb{M}_{ m - 1 }
  +
  \left( 
    1 - \mathbb{Y} 
  \right)
  \operatorname{Pow}_{ 2 }\!\left(
    \frac{ 1 }{ J_m }
    \sum_{ j = 1 }^{ J_m }
      \Phi^{ m - 1 , j }_{
        \mathbb{S}_m
      }( 
        \Theta_{ m - 1 } 
      ) 
  \right)
  .
\end{split}
\end{equation}
\end{example}

\subsection{Geometric Brownian motion}
\label{sec:GBM}

\begin{example}
\label{ex:geometricBM}
Assume the setting in Section~\ref{sec:general_case}, 
let $ \bar{\mu}, \bar{\sigma} \in \R $,
and assume 
for all 
$ s, t \in [0,T] $,
$ 
  x = ( x_1, \dots, x_d )
$,
$
  w = ( w_1, \dots, w_d ) \in \R^d 
$
that
\begin{equation}
  \Upsilon( s, t, x, w )
  =
  \exp\!\left(
    \left( \bar{\mu} - \frac{ \bar{\sigma}^2 }{ 2 } \right) ( t - s )
  \right)
  \exp\!\left(
    \bar{\sigma} 
    \operatorname{diag}_{ \R^{ d \times d } }( w_1, \dots, w_d )
  \right)
  x
  .
\end{equation}
Then it holds 
for all 
$ m, j \in \N_0 $,
$ n \in \{ 0, 1, \dots, N \} $
that
\begin{equation}
  \mathcal{X}^{
    \theta, m, j 
  }_n
  =
  \exp\!\left(
    \left( 
      \bar{\mu} - \frac{ \bar{\sigma}^2 }{ 2 } 
    \right) 
    t_n
    \operatorname{Id}_{ \R^d }
    +
    \bar{\sigma} 
    \operatorname{diag}_{ \R^{ d \times d } }( W_{ t_n }^{ m, j } )
  \right)
  \xi
  .
\end{equation}
\end{example}

In the setting of Example~\ref{ex:geometricBM} 
we consider under suitable further hypotheses 
(cf.\ Subsection~\ref{subsec:borrowlend} above)
for every sufficiently large 
$ m \in \N_0 $
the random variable
$ \mathcal{U}^{ \Theta_m } $
as an approximation of $ u(0,\xi) $
where
$ 
  u \colon [0,T] \times \R^d \to \R^k 
$
is a suitable solution of the PDE
\begin{multline}
   \tfrac{ \partial u}{ \partial t } ( t, x )
  +
  \tfrac{\bar{\sigma}^2}{ 2 }
  \smallsum\limits_{i=1}^d
  | x_i |^2
  \, 
  \big( 
    \tfrac{ \partial^2 u}{ \partial x^2_i } 
  \big)(t,x)
  +
  \bar{\mu} \sum\limits_{i=1}^d x_i 
  \, \big(\tfrac{\partial u}{\partial x_i}\big)(t,x)
\\ 
  + 
  f\big( 
    t, x, 
    u(t,x),
    \bar{\sigma} 
    \,
    ( \tfrac{ \partial u}{ \partial x } )( t, x )
    \operatorname{diag}_{ \R^{ d \times d } }(x_1, \dots, x_d) 
  \big)
  = 0
\end{multline}
with 
$ u(T,x) = g(x) $
for $ t \in [0,T] $, $ x = ( x_1, \dots, x_d ) \in \R^d $.

\subsection{Euler-Maruyama scheme}
\label{sec:EM}

\begin{example}
\label{ex:EM}
Assume the setting in Section~\ref{sec:general_case}, 
let 
$ \mu \colon [0,T] \times \R^d \to \R^d $
and
$ \sigma \colon [0,T] \times \R^d \to \R^d $
be functions, 
and assume 
for all 
$ s, t \in [0,T] $,
$ 
  x, w \in \R^d 
$
that
\begin{equation}
  \Upsilon( s, t, x, w )
  =
  x
  +
  \mu( s, x )
  \, ( t - s )
  +
  \sigma( s, x )
  \, 
  w
  .
\end{equation}
Then it holds 
for all 
$ m, j \in \N_0 $,
$ n \in \{ 0, 1, \dots, N - 1 \} $
that
\begin{equation}
  \mathcal{X}^{
    m, j 
  }_n
  =
  \mathcal{X}^{
    m, j 
  }_n
  +
  \mu( t_n, 
    \mathcal{X}^{
      m, j 
    }_n
  )
  \, ( t_{ n + 1 } - t_n )
  +
  \sigma( t_n, 
    \mathcal{X}^{
      m, j 
    }_n
  )
  \, ( W_{ t_{ n + 1 } } - W_{ t_n } )
  .
\end{equation}
\end{example}

In the setting of Example~\ref{ex:EM} 
we consider under suitable further hypotheses 
for every sufficiently large 
$ m \in \N_0 $
the random variable
$ \mathcal{U}^{ \Theta_m } $
as an approximation of $ u(0,\xi) $
where
$ 
  u \colon [0,T] \times \R^d \to \R^k 
$
is a suitable solution of the PDE
\begin{multline}
  \tfrac{ \partial u}{ \partial t } ( t, x )
  +
  \tfrac{ 1 }{ 2 }
  \smallsum\limits_{ j = 1 }^d
  ( 
    \frac{ \partial^2 u}{ \partial x^2 } 
  )( t, x )\big(
    \sigma( 
      t, x
    )
    \,
    e^{ (d) }_j
    ,
    \sigma( 
      t, x
    )
    \,
    e^{ (d) }_j
  \big)
  +
    ( 
      \tfrac{ \partial u}{ \partial x } 
    )( t, x )
    \,
    \mu( t, x )
\\ 
  +
  f\big( 
    t, x, u(t,x), 
    ( \tfrac{ \partial u}{ \partial x } )( t, x ) 
    \,
    \sigma( t, x ) 
  \big)
  = 
  0 
\end{multline}
with 
$ u(T,x) = g(x) $,
$ e^{ (d) }_1 = (1,0,\dots,0) $,
$ \dots $,
$ e^{ (d) }_d = (0,\dots,0,1) \in \R^d $
for $ t \in [0,T] $, $ x = ( x_1, \dots, x_d ) \in \R^d $
(cf.\ \eqref{eq:PDE} in Section~\ref{sec:derivation} above).

\section{Appendix B: {\sc Python} and {\sc Matlab} source codes}

\subsection{{\sc Python} source code for an implementation of the deep BSDE solver 
in the case of the Allen-Cahn PDE~\eqref{eq:PDE_allencahn} in Subsection~\ref{subsec:allencahn}}
\label{sec:Python_code}

\begin{lstlisting}[caption={A {\sc Python} code for the 
deep BSDE solver in 
Subsection~\ref{sec:general_case} in the case 
of the PDE~\eqref{eq:PDE_allencahn}.},
label=code:deepPDEmethod]
import time
import math
import tensorflow as tf
import numpy as np
from tensorflow.python.training.moving_averages \
    import assign_moving_average
from scipy.stats import multivariate_normal as normal
from tensorflow.python.ops import control_flow_ops
from tensorflow import random_normal_initializer as norm_init
from tensorflow import random_uniform_initializer as unif_init
from tensorflow import constant_initializer as const_init

class SolveAllenCahn(object):
    """The fully-connected neural network model."""
    def __init__(self, sess):
        self.sess = sess
        # parameters for the PDE
        self.d = 100
        self.T = 0.3
        # parameters for the algorithm
        self.n_time = 20
        self.n_layer = 4
        self.n_neuron = [self.d, self.d+10, self.d+10, self.d]
        self.batch_size = 64
        self.valid_size = 256
        self.n_maxstep = 4000
        self.n_displaystep = 100
        self.learning_rate = 5e-4
        self.Yini = [0.3, 0.6]
        # some basic constants and variables
        self.h = (self.T+0.0)/self.n_time
        self.sqrth = math.sqrt(self.h)
        self.t_stamp = np.arange(0, self.n_time)*self.h
        self._extra_train_ops = []

    def train(self):
        start_time = time.time()
        # train operations
        self.global_step = \
            tf.get_variable('global_step', [],
                            initializer=tf.constant_initializer(1),
                            trainable=False, dtype=tf.int32)
        trainable_vars = tf.trainable_variables()
        grads = tf.gradients(self.loss, trainable_vars)
        optimizer = tf.train.AdamOptimizer(self.learning_rate)
        apply_op = \
            optimizer.apply_gradients(zip(grads, trainable_vars),
                                      global_step=self.global_step)
        train_ops = [apply_op] + self._extra_train_ops
        self.train_op = tf.group(*train_ops)
        self.loss_history = []
        self.init_history = []
        # for validation
        dW_valid, X_valid = self.sample_path(self.valid_size)
        feed_dict_valid = {self.dW: dW_valid,
                           self.X: X_valid,
                           self.is_training: False}
        # initialization
        step = 1
        self.sess.run(tf.global_variables_initializer())
        temp_loss = self.sess.run(self.loss,
                                  feed_dict=feed_dict_valid)
        temp_init = self.Y0.eval()[0]
        self.loss_history.append(temp_loss)
        self.init_history.append(temp_init)
        print "step: %5u,  loss: %.4e,  " % \
                      (0, temp_loss) + \
                      "Y0: %.4e,  runtime: %4u" % \
                      (temp_init, time.time()-start_time+self.t_bd)
        # begin sgd iteration
        for _ in range(self.n_maxstep+1):
            step = self.sess.run(self.global_step)
            dW_train, X_train = self.sample_path(self.batch_size)
            self.sess.run(self.train_op, 
                          feed_dict={self.dW: dW_train,
                                     self.X: X_train,
                                     self.is_training: True})
            if step % self.n_displaystep == 0:
                temp_loss = self.sess.run(self.loss, 
                                          feed_dict=feed_dict_valid)
                temp_init = self.Y0.eval()[0]
                self.loss_history.append(temp_loss)
                self.init_history.append(temp_init)
                print "step: %5u,  loss: %.4e,  " % \
                      (step, temp_loss) + \
                      "Y0: %.4e,  runtime: %4u" % \
                      (temp_init, time.time()-start_time+self.t_bd)
            step += 1
        end_time = time.time()
        print "running time: %.3f s" % \
            (end_time-start_time+self.t_bd)

    def build(self):
        start_time = time.time()
        # build the whole network by stacking subnetworks
        self.dW = tf.placeholder(tf.float64, 
                                 [None, self.d, self.n_time],
                                 name='dW')
        self.X = tf.placeholder(tf.float64, 
                                [None, self.d, self.n_time+1],
                                name='X')
        self.is_training = tf.placeholder(tf.bool)
        self.Y0 = tf.Variable(tf.random_uniform([1],
                                                minval=self.Yini[0],
                                                maxval=self.Yini[1],
                                                dtype=tf.float64));
        self.Z0 = tf.Variable(tf.random_uniform([1, self.d],
                                                minval=-.1,
                                                maxval=.1,
                                                dtype=tf.float64))
        self.allones = \
            tf.ones(shape=tf.pack([tf.shape(self.dW)[0], 1]),
                    dtype=tf.float64)
        Y = self.allones * self.Y0
        Z = tf.matmul(self.allones, self.Z0)
        with tf.variable_scope('forward'):
            for t in xrange(0, self.n_time-1):
                Y = Y - self.f_tf(self.t_stamp[t],
                                  self.X[:, :, t], Y, Z)*self.h
                Y = Y + tf.reduce_sum(Z*self.dW[:, :, t], 1,
                                      keep_dims=True)
                Z = self._one_time_net(self.X[:, :, t+1],
                                       str(t+1))/self.d
            # terminal time
            Y = Y - self.f_tf(self.t_stamp[self.n_time-1],
                              self.X[:, :, self.n_time-1],
                              Y, Z)*self.h
            Y = Y + tf.reduce_sum(Z*self.dW[:, :, self.n_time-1], 1,
                                  keep_dims=True)
            term_delta = Y - self.g_tf(self.T,
                                       self.X[:, :, self.n_time])
            self.clipped_delta = \
                tf.clip_by_value(term_delta, -50.0, 50.0)
            self.loss = tf.reduce_mean(self.clipped_delta**2)
        self.t_bd = time.time()-start_time

    def sample_path(self, n_sample):
        dW_sample = np.zeros([n_sample, self.d, self.n_time])
        X_sample = np.zeros([n_sample, self.d, self.n_time+1])
        for i in xrange(self.n_time):
            dW_sample[:, :, i] = \
                np.reshape(normal.rvs(mean=np.zeros(self.d),
                                      cov=1,
                                      size=n_sample)*self.sqrth,
                           (n_sample, self.d))
            X_sample[:, :, i+1] = X_sample[:, :, i] + \
                                  np.sqrt(2) * dW_sample[:, :, i]
        return dW_sample, X_sample

    def f_tf(self, t, X, Y, Z):
        # nonlinear term
        return Y-tf.pow(Y, 3)

    def g_tf(self, t, X):
        # terminal conditions
        return 0.5/(1 + 0.2*tf.reduce_sum(X**2, 1, keep_dims=True))

    def _one_time_net(self, x, name):
        with tf.variable_scope(name):
            x_norm = self._batch_norm(x, name='layer0_normal')
            layer1 = self._one_layer(x_norm, self.n_neuron[1],
                                     name='layer1')
            layer2 = self._one_layer(layer1, self.n_neuron[2],
                                     name='layer2')
            z = self._one_layer(layer2, self.n_neuron[3],
                                activation_fn=None, name='final')
        return z

    def _one_layer(self, input_, out_sz, 
                   activation_fn=tf.nn.relu,
                   std=5.0, name='linear'):
        with tf.variable_scope(name):
            shape = input_.get_shape().as_list()
            w = tf.get_variable('Matrix',
                                [shape[1], out_sz], tf.float64,
                                norm_init(stddev= \
                                    std/np.sqrt(shape[1]+out_sz)))
            hidden = tf.matmul(input_, w)
            hidden_bn = self._batch_norm(hidden, name='normal')
        if activation_fn != None:
            return activation_fn(hidden_bn)
        else:
            return hidden_bn

    def _batch_norm(self, x, name):
        """Batch normalization"""
        with tf.variable_scope(name):
            params_shape = [x.get_shape()[-1]]
            beta = tf.get_variable('beta', params_shape,
                                   tf.float64,
                                   norm_init(0.0, stddev=0.1,
                                             dtype=tf.float64))
            gamma = tf.get_variable('gamma', params_shape,
                                    tf.float64,
                                    unif_init(0.1, 0.5,
                                              dtype=tf.float64))    
            mv_mean = tf.get_variable('moving_mean',
                                      params_shape,
                                      tf.float64,
                                      const_init(0.0, tf.float64),
                                      trainable=False)
            mv_var = tf.get_variable('moving_variance',
                                     params_shape,
                                     tf.float64,
                                     const_init(1.0, tf.float64),
                                     trainable=False)
            # These ops will only be preformed when training
            mean, variance = tf.nn.moments(x, [0], name='moments')
            self._extra_train_ops.append(\
                assign_moving_average(mv_mean, mean, 0.99))
            self._extra_train_ops.append(\
                assign_moving_average(mv_var, variance, 0.99))
            mean, variance = \
                control_flow_ops.cond(self.is_training,
                                      lambda: (mean, variance),
                                      lambda: (mv_mean, mv_var))
            y = tf.nn.batch_normalization(x, mean, variance, 
                                          beta, gamma, 1e-6)
            y.set_shape(x.get_shape())
            return y

def main():
    tf.reset_default_graph()
    with tf.Session() as sess:
        tf.set_random_seed(1)
        print "Begin to solve Allen-Cahn equation"
        model = SolveAllenCahn(sess)
        model.build()
        model.train()
        output = np.zeros((len(model.init_history), 3))
        output[:, 0] = np.arange(len(model.init_history)) \
                          * model.n_displaystep
        output[:, 1] = model.loss_history
        output[:, 2] = model.init_history
        np.savetxt("./AllenCahn_d100.csv",
                   output,
                   fmt=['%d', '%.5e', '%.5e'],
                   delimiter=",",
                   header="step, loss function, " + \
                          "target value, runtime",
                   comments='')

if __name__ == '__main__':
    np.random.seed(1)
    main()
    
\end{lstlisting}

\subsection{{\sc Matlab} source code for the Branching diffusion method used in Subsection~\ref{subsec:allencahn}}

\begin{lstlisting}[caption={A {\sc Matlab} code for the 
Branching diffusion method in the case of 
the PDE~\eqref{eq:PDE_allencahn} based on 
$ M = 10^7 $ independent realizations.}, 
label=code:Branching]
function Branching_Matlab()
% Parameters for the model
  T = 0.3; t0 = 0; x0 = 0; d = 100; m = d; 
  mu = zeros(d,1); sigma = eye(d)*sqrt(2);
  a = [0 2 0 -1]';
  g = @(x) 1./(1+0.2*norm(x)^2)*1/2;

% Parameters for the algorithm
  rng('default'); M = 10^7; beta = 1; p = [0 0.5 0 0.5]';

% Branching method
  tic;
  [mn,sd] = MC_BM( mu, sigma, beta, p, a, t0, x0, T, g, M );
  runtime = toc;

% Output
  disp(['Terminal condition: u(T,x0) = ' num2str(g(x0)) ';']);
  disp(['Branching method: u(0,x0) ~ ' num2str(mn) ';']);
  disp(['Estimated standard deviation: ' num2str(sd) ';']);
  disp(['Estimated L2-appr. error = ' num2str(sd/sqrt(M)) ';']);
  disp(['Elapsed runtime = ' num2str(runtime) ';']);
end

function [mn,sd] = MC_BM(mu, sigma, beta, p, a, t0, x0, T, g, M)
  mn = 0; sd = 0;
  for m=1:M
    result = BM_Eval(mu, sigma, beta, p, a, t0, x0, T, g);
    mn = mn + result;
    sd = sd + result^2;
  end 
  mn = mn/M; sd = sqrt( (sd - mn^2/M)/M );
end

function result = BM_Eval(mu, sigma, beta, p, a, t0, x0, T, g)
  bp = BP(mu, sigma, beta, p, t0, x0, T);
  result = 1;
  for k=1:size(bp{1},2)
    result = result * g( bp{1}(:,k) );
  end  
  if norm(a-p) > 0
    for k=1:length(a)
      if p(k) > 0
        result = result * ( a(k)/p(k) )^( bp{2}(k) );
      elseif a(k) ~= 0
        error('a(k) zero but p(k) non-zero');
      end
    end
  end
end

function bp = BP(mu, sigma, beta, p, t0, x0, T)
  bp = cell(2,1);
  bp{2} = p*0;
  tau = exprnd(1/beta);
  new_t0 = min( tau + t0, T );
  delta_t = new_t0 - t0;
  m = size(sigma,2);
  new_x0 = x0 + mu*delta_t + sigma*sqrt(delta_t)*randn(m,1);
  if tau >= T - t0
    bp{1} = new_x0;
  else 
    [tmp,nonlinearity] = max(mnrnd(1,p));
    bp{2}(nonlinearity) = bp{2}(nonlinearity) + 1;
    for k=1:nonlinearity-1
      tmp = BP(mu, sigma, beta, p, new_t0, new_x0, T);
      bp{1} = [ bp{1} tmp{1} ];
      bp{2} = bp{2} + tmp{2};
    end
  end
end
\end{lstlisting}

\subsection{{\sc Matlab} source code for the classical Monte Carlo method used in Subsection~\ref{subsec:HJB}}

\begin{lstlisting}[caption={A {\sc Matlab} code for a Monte Carlo method related 
to the PDE~\eqref{eq:PDE_HJB} based on 
$ M = 10^7 $ independent realizations.}, 
label=code:linearMC]
function MonteCarlo_Matlab()
  rng('default');

% Parameters for the model
  d = 100;
  g = @(x) log( (1+norm(x)^2)/2 );
  T = 1;
  M = 10^7;
  t = 0;

% Classical Monte Carlo
  tic;
  MC = 0;
  for m=1:M
    dW = randn(1,d)*sqrt(T-t);
    MC = MC + exp( - g( dW * sqrt(2) ) );
  end
  MC = -log(MC/M);
  runtime = toc;

% Output
  disp(['Solution: u(T,0) = ' num2str(g(0)) ';']);
  disp(['Solution: u(0,0) = ' num2str(MC) ';']);
  disp(['Time = ' num2str(runtime) ';']);
end

\end{lstlisting}

\subsubsection*{Acknowledgements}

Christian Beck and Sebastian Becker are gratefully acknowledged 
for useful suggestions regarding the implementation of the deep BSDE solver.
This project has been partially supported through 
the Major Program of NNSFC under grant 91130005,
the research grant ONR N00014-13-1-0338, 
and the research grant DOE DE-SC0009248.

\bibliographystyle{acm}
\bibliography{bibfile}

\end{document}